\documentclass[a4paper,reqno]{amsart}
\usepackage{geometry}
\usepackage{amsfonts}
\usepackage{amsmath,amssymb,amsthm,amsxtra}
\usepackage{float}
\usepackage[colorlinks, linkcolor=blue!72,anchorcolor=orange,
    citecolor=blue,urlcolor=Emerald, bookmarksopen,bookmarksdepth=2]{hyperref}
\usepackage[usenames,dvipsnames]{xcolor}
\usepackage{enumitem}
\usepackage{geometry,array}
\usepackage{graphicx}
\usepackage{subfigure}
\usepackage{bookmark}
\usepackage{tikz}
\usepackage{url}

\usetikzlibrary{matrix,positioning,decorations.markings,arrows,decorations.pathmorphing,	
    backgrounds,fit,positioning,shapes.symbols,chains,shadings,fadings,calc}
\tikzset{->-/.style={decoration={  markings,  mark=at position #1 with
    {\arrow{>}}},postaction={decorate}}}
\tikzset{-<-/.style={decoration={  markings,  mark=at position #1 with
    {\arrow{<}}},postaction={decorate}}}

\usepackage[all]{xy}

\usepackage{color}
\def\red{\color{red}}

\numberwithin{equation}{section}

\theoremstyle{plain}

\newtheorem{thm}{Theorem}[section]
\newtheorem{cor}[thm]{Corollary}
\newtheorem{lem}[thm]{Lemma}
\newtheorem{prop}[thm]{Proposition}

\theoremstyle{definition}
\newtheorem{defn}[thm]{Definition}
\newtheorem{exm}[thm]{Example}
\newtheorem{rem}[thm]{Remark}

\newdir{ >}{{}*!/-5pt/\dir{>}}

\newcommand{\Hom}{\operatorname{Hom}\nolimits}

\newcommand{\op}{\operatorname{op}\nolimits}

\newcommand{\Id}{\operatorname{Id}\nolimits}

\renewcommand{\mod}{\mathsf{mod}\hspace{.01in}}

\newcommand{\M}{\mathcal M}

\newcommand{\B}{\mathcal B}

\newcommand{\oB}{\overline{\B}}
\newcommand{\U}{\mathcal U}
\newcommand{\V}{\mathcal V}

\newcommand{\W}{\mathcal W}

\newcommand{\s}{\mathcal S}
\newcommand{\T}{\mathcal T}

\newcommand{\I}{\mathcal I}
\newcommand{\D}{\mathcal D}

\newcommand{\K}{\mathcal K}
\newcommand{\N}{\mathcal N}
\newcommand{\R}{\mathcal R}
\newcommand{\X}{\mathcal X}
\newcommand{\Y}{\mathcal Y}
\newcommand{\Z}{\mathcal Z}
\newcommand{\C}{\mathcal C}

\newcommand{\EE}{\mathbb E}

\newcommand{\svecv}[2]{\left(\begin{smallmatrix}
      #1 \\
      #2
    \end{smallmatrix}\right)}

\newcommand{\svech}[2]{\left(\begin{smallmatrix}
      #1 & #2
\end{smallmatrix}\right)}

\renewcommand{\emph}{\textit}
\renewcommand{\phi}{\varphi}

\begin{document}

\title{Extriangulated ideal quotients and Gabriel-Zisman localizations}

\thanks{Yu Liu was supported by the National Natural Science Foundation of China (Grant No. 12171397). Panyue Zhou was supported by the National Natural Science Foundation of China (Grant No. 12371034) and by the Hunan Provincial Natural Science Foundation of China (Grant No. 2023JJ30008). }

\author{Yu Liu and Panyue Zhou$^\ast$}
\address{School of Mathematics and Statistics, Shaanxi Normal University, 710062 Xi'an, Shaanxi, P. R. China}
\email{liuyu86@swjtu.edu.cn}
\address{School of Mathematics and Statistics, Changsha University of Science and Technology, 410114 Changsha, Hunan,  P. R. China}
\email{panyuezhou@163.com}
\thanks{$^\ast$Corresponding author. \hspace{1em}
The authors would like to thank Professor Bin Zhu for helpful discussions.}

\begin{abstract}
Let $(\B,\EE,\mathfrak{s})$ be an extriangulated category and $\s$ be an extension closed subcategory of $\B$.
In this article, we prove that the Gabriel-Zisman localization $\B/\s$ can be realized as
an ideal quotient inside $\B$ when $\s$ satisfies some mild conditions. The ideal quotient is an extriangulated category. We show that the equivalence between the ideal quotient and the localization preserves the extriangulated category structure. We also discuss the relations of our results with Hovey twin cotorsion pairs and Verdier quotients.

\end{abstract}
\maketitle

\section{Introduction}
 Triangulated categories were introduced in the mid 1960's by Verdier \cite{V}.
Having their origins in algebraic geometry and
algebraic topology, triangulated categories have by now become indispensable
in many different areas of mathematics. The Verdier quotient $\mathcal T/\s$ of a triangulated category $\mathcal T$ by a triangulated subcategory $\s$ is defined by a universal property with respect to triangulated functors out of $\mathcal T$. On the other hand, $\mathcal T/\s$ is in fact a localization of $\mathcal T$, which means
 it is obtained from $\mathcal T$ by formally inverting a class of morphisms.
For example, derived categories are certain Verdier quotients of homotopy categories.
Localization is a process of adding formal inverses to an algebraic structure known as Gabriel-Zisman localisation \cite{GZ}, morphisms in the new category can be regarded as compositions of the original morphisms and the formal inverses
that were added. This makes Verdier quotients a bit hard to understand since
taking Verdier quotients drastically change morphisms.

Iyama and Yang \cite{IYa2} gave a sufficient condition for a Verdier quotient $\mathcal T /\s$ of a triangulated category $\mathcal T$ by a thick subcategory $\s$ to be realized inside of $\mathcal T$ as an ideal quotient.
Concretely speaking, they assume that $\mathcal T$ and $\s$ satisfy the following conditions:
\begin{itemize}
\item[(T0)]  $\mathcal T$ is a triangulated category with a shift functor $[1]$, and $\s$ is a thick subcategory of $\mathcal T$.  Denote by
$\mathcal U:=\mathcal T/\s$ the Verdier quotient (which is a triangulated category).

\item[(T1)]   $\s$ has a torsion pair $(\X,\Y)$.

\item[(T2)] $(\X,\X^{\perp})$ and $({^\perp\Y},\Y)$ form two torsion pairs in $\mathcal T$,
where $\X^{\perp}=:\{T\in\mathcal T\mid {\rm Hom}_{\mathcal T}(\X,T)=0\}$
and ${^{\perp}\Y}=:\{T\in\mathcal T\mid {\rm Hom}_{\mathcal T}(T,\Y)=0\}$.
\end{itemize}
Define two full subcategories of $\mathcal T$:
$$\mathcal Z:=\X^{\perp}\cap{^\perp\Y[1]}\hspace{2mm}\mbox{and}\hspace{2mm}\mathcal M:=\X[1]\cap \Y.$$
Denote by $\mathcal Z/[\mathcal M]$ the ideal quotient category of $\Z$ by $\mathcal M${\red .} Iyama and Yang \cite{IYa2} realized the Verdier quotient $\mathcal U=\mathcal T /\s$ as the ideal quotient
$\mathcal Z/[\mathcal M]$.

\begin{thm}\cite[Theorem 1.1]{IYa2}\label{1.1}
With the assumptions {\rm (T0), (T1)} and {\rm (T2)} as above,
 the composition $\mathcal Z\hookrightarrow \mathcal T\to \mathcal U$
of natural functors induces an equivalence of additive categories:
$\mathcal Z/[\mathcal M]\xrightarrow{\hspace{1mm}\simeq\hspace{1mm}}\mathcal U$.
\end{thm}

When the condition (T1) is replaced by a special case:
$$ \text{(T1}')\quad (\X,\Y) \text{ is a co-}t \text{-structure in } \s,$$
the equivalence above becomes a \textbf{triangle} equivalence,
see \cite[Theorem 1.2]{IYa2}. There are many examples of this realization, such as Buchweitz's equivalence \cite{Bu,KV,R} between the singularity category of an Iwanaga-Gorenstein ring and the stable category of Cohen-Macaulay modules over the ring. More related researches can be found in \cite{Am,BOJ,C,CZ,G,IYa1,K1,OPS,ZH}.

Recently, the notion of an extriangulated category   was introduced by Nakaoka and Palu \cite{NP} as a
simultaneous generalization of exact categories and triangulated categories.
Exact categories and extension closed subcategories of a triangulated category
are extriangulated categories, hence
 many results on exact categories and triangulated categories can be unified
in the framework of an extriangulated category. There are also plenty of examples of
extriangulated categories which are neither exact nor triangulated, see \cite{NP,ZZ,HZZ,ZhZ}.

In \cite{NP}, Nakaoka and Palu gave a bijection between Hovey twin cotorsion pairs and admissible model structures.
Let $(\B,\EE,\mathfrak{s})$ be an extriangulated category satisfying some mild conditions, and
$((\s,\T), (\U,\V))$ be a Hovey twin cotorsion pair in $\B$.
Define two full subcategories of $\B$:
$$\mathcal M:=\U\cap\V\hspace{2mm}\mbox{and}\hspace{2mm}\mathcal Z:=\mathcal T\cap\mathcal U.$$
Consider the following classes of morphisms:
\begin{itemize}
\item $wFib:=$ the class of deflations $f$ with ${\rm CoCone}(f)\in\mathcal V$;

\item ${wCof}:=$ the class of inflations $g$ with ${\rm Cone}(g)\in\mathcal S$;

\item $\mathbb{W}:=wFib\circ wCof$.
\end{itemize}

 They showed the following theorem.

\begin{thm}{\rm \cite[Corollary 5.25, Theorem 6.20]{NP}}\label{np}\label{1.2}
The composition of the canonical inclusion $i\colon\Z\hookrightarrow \B$ and
the Gabriel-Zisman localization $\ell\colon \B\to \B[\mathbb{W}^{-1}]$
induces an equivalence $\overline{\ell}:\mathcal Z/[\mathcal M]\xrightarrow{\hspace{1mm}\simeq\hspace{1mm}}\B[\mathbb{W}^{-1}]$,
which is depicted as follows:
$$
\xymatrix{
\Z~\ar[d]_\pi\ar@{^{(}->}[r]^{i}&\B\ar[r]^{\ell\quad\;\;}&\B[\mathbb{W}^{-1}].\\
\Z/\mathcal M\ar@{..>}[rru]_{\overline{\ell}}&&
}
$$
Moreover, $\B[\mathbb{W}^{-1}]$ is a triangulated category.
\end{thm}


Note that $(\mathcal X, \mathcal Y)$ is a torsion pair in a triangulated category $\T$ in the sense of Iyama and  Yoshino \cite{IYo}
if and only if $(\mathcal X[1],\mathcal Y)$  is a cotorsion pair in $\T$ in the sense of Nakaoka \cite{N}. In this article, by using cotorsion pairs, we will develop a theory which shows that similar equivalences between localizations and ideal quotients exist under a more general setting. Let $(\B,\EE,\mathfrak{s})$ be an extriangulated category. This paper is dedicated to
the Gabriel-Zisman localization $\B/\s$ of an extriangulated category $(\B,\EE,\mathfrak{s})$ with respect to  an extension closed subcategory $\s$ to be realized as an idea quotient inside $\B$.
Our main result is the following.

\begin{thm}\label{main} {\rm (see Theorem \ref{main5} and Sections 3 and 4 for details)}
Let $k$ be a field, $(\B,\EE,\mathfrak{s})$ be a Krull-Schmidt, Hom-finite, $k$-linear extriangulated category  satisfying condition {\rm (WIC)},
and $\s$ be an extension closed  full subcategory of $\B$. Let $\X$ and $\Y$ be full subcategories of $\B$ satisfying the following conditions:
\begin{itemize}
\item[\rm (B1)] $((\X,{\X^{\bot_1}}),({^{\bot_1}}\Y,\Y))$ is a twin cotorsion pair in $\B$,
where ${\X^{\bot_1}}=:\{B\in\B\mid \EE(\X,B)=0\}$ and ${^{\bot_1}}\Y=:\{B\in\B\mid \EE(B,\Y)=0\}$.
\smallskip

\item[\rm (B2)] $(\X,\Y)$ is a cotorsion pair in $\s$.
\end{itemize}
Let
$$\Z:=\X^{\bot_1}\cap{^{\bot_1}}\Y\hspace{2mm}\mbox{and}\hspace{2mm}\W:=\X\cap\Y.$$
Assume that $\s$ is closed under taking cones, which means
$$\text{for any }\EE\text{-triangle }A\to B\to C\dashrightarrow \text{, if }A,B\in \s, \text{ then we have }C\in\s.$$
Then the Gabriel-Zisman localization $\B/\s$ (see Definition \ref{lo} for details) can be realized as the ideal quotient $\Z/[\W]$,
that is, there exists an equivalence
 $F:\mathcal Z/[\mathcal W]\xrightarrow{\hspace{1mm}\simeq\hspace{1mm}}\B/\s$.
Moreover, when $\B$ has enough projectives and enough injectives,
\begin{itemize}
\item[\rm (1)] if $\mathcal S$ is a thick subcategory of $\B$, then $F$ becomes an \textbf{extriangle} equivalence (see Definition \ref{exeq} for details);
\item[\rm (2)] if $(\X, \X^{\bot_1})$, $({{^{\bot_1}}\Y},\Y)$ are hereditary cotorsion pairs, then $F$ becomes a \textbf{triangle} equivalence.
\end{itemize}
\end{thm}

\begin{rem}
In Theorem \ref{1.1} and Theorem \ref{1.2}, the equivalences between ideal quotients and localizations are additive, although the localizations are triangulated and the ideal quotients are extriangulated. Theorem \ref{main} points out that the equivalences have better properties.
\end{rem}

%

This article is organized as follows. In Section 2, we review some elementary concepts and properties of extriangulated categories. In Section 3, we show the main result of this article: the existence of equivalences between localizations and quotient categories. In Section 4, we show that under certain conditions, the equivalence has better properties.

\section{Preliminaries}

Let us briefly recall the definition and some basic properties of extriangulated categories.
For more details, see \cite[Section 2,3]{NP}.

Let $\B$ be an additive category equipped with an additive bifunctor
$$\mathbb{E}: \B^{\rm op}\times \B\rightarrow {\rm Ab},$$
where ${\rm Ab}$ is the category of abelian groups. For any objects $A, C\in\B$, an element $\delta\in \mathbb{E}(C,A)$ is called an $\mathbb{E}$-extension.
Let $\mathfrak{s}$ be a correspondence which associates with an equivalence class of sequences
$$\mathfrak{s}(\delta)=\xymatrix@C=0.8cm{[A\ar[r]^x &B\ar[r]^y&C]}$$
to any $\mathbb{E}$-extension $\delta\in\mathbb{E}(C, A)$. This $\mathfrak{s}$ is called a {\it realization} of $\mathbb{E}$, if it makes the diagrams in \cite[Definition 2.9]{NP} commutative.
 A triplet $(\B, \mathbb{E}, \mathfrak{s})$ is called an {\it extriangulated category} if:
\begin{enumerate}
\setlength{\itemsep}{2.5pt}
\item $\mathbb{E}\colon\B^{\rm op}\times \B\rightarrow \rm{Ab}$ is an additive bifunctor.

\item $\mathfrak{s}$ is an additive realization of $\mathbb{E}$.

\item $\mathbb{E}$ and $\mathfrak{s}$  satisfy some `additivity' and `compatibility' conditions in \cite[Definition 2.12]{NP}.
 \end{enumerate}
\smallskip

We collect some basic concepts  which will be used later.

\begin{defn}\label{dein}
Let $(\B,\EE,\mathfrak{s})$ be an extriangulated category.
\begin{itemize}
\setlength{\itemsep}{2.5pt}
\item[{\rm (1)}] If a sequence $A\xrightarrow{~x~}B\xrightarrow{~y~}C$ realizes $\delta\in\mathbb{E}(C,A)$, we call the pair $( A\xrightarrow{~x~}B\xrightarrow{~y~}C,\delta)$ an {\it $\EE$-triangle}, and write it in the following way:
$$A\overset{x}{\longrightarrow}B\overset{y}{\longrightarrow}C\overset{\delta}{\dashrightarrow}.$$
We usually do not write this $``\delta"$ if it is not used in the argument.

\item[{\rm (2)}] An object $P\in\B$ is called {\it projective} if
for any $\EE$-triangle $A\overset{x}{\longrightarrow}B\overset{y}{\longrightarrow}C \dashrightarrow$ and any morphism $c:P\to C$, there exists $b:P\to B$ satisfying $yb=c$.
We denote the subcategory of projective objects by  $\mathcal P$. Dually, the subcategory of injective objects is denoted by $\I$.

\item[{\rm (3)}] We say that $\B$ {\it has enough projectives} if
for any object $C\in\B$, there exists an $\EE$-triangle
$A\to P\to C\dashrightarrow$
with $P\in\mathcal P$. Dually we can define  {\it having enough injectives}.

\item[{\rm (4)}] Let $\s$ be a subcategory of $\B$. We say $\s$ is {\it extension closed}
if  in any $\EE$-triangle $A\rightarrow B\rightarrow C\dashrightarrow$ with $A,C\in\s$, we have $B\in\s$.

\item[{\rm (5)}] Any $\EE$-triangle $A\to B\to C\dashrightarrow$ induces the following long exact sequences:
\begin{align*}
\Hom_\B(X,A)\to \Hom_\B(X,B)\to \Hom_\B(X,C)\to \EE(X,A)\to \EE(X,B)\to \EE(X,C);\\
\Hom_\B(C,X)\to \Hom_\B(B,X)\to \Hom_\B(A,X)\to \EE(C,X)\to \EE(B,X)\to \EE(A,X)
\end{align*}
where $X$ is an arbitrary object in $\B$.
\end{itemize}
\end{defn}

\begin{rem}
Any extension closed subcategory $\M$ of an extriangulated category $(\B,\EE,\mathfrak{s})$ has a natural extriangulated category structure $(\M,\EE|_{\M},\mathfrak{s}|_{\M})$  which inherits from $(\B,\EE,\mathfrak{s})$,
where $\EE|_{\M}$ is the restriction of $\EE$ onto ${\M^{\rm op}\times \M}$
and  $\mathfrak{s}|_{\M}$ is the restriction of $\mathfrak{s}$.
\end{rem}

In this paper, let $k$ be a field and $(\B,\mathbb{E},\mathfrak{s})$ be a Krull-Schmidt, Hom-finite, $k$-linear extriangulated category. Let $\mathcal P$ (resp. $\mathcal I$) be the subcategory of projective (resp. injective) objects.
When we say that $\C$ is a subcategory of $\B$, we always assume that $\C$ is full and closed under isomorphisms.

In this paper, the cotorsion pair is the main tool we use.

\begin{defn}\cite[Definition 2.1]{NP}
Let $\U$ and $\V$ be two subcategories of $\B$ which are closed under direct summands. We call $(\U,\V)$ a \emph{cotorsion pair} if it satisfies the following conditions:
\begin{itemize}
\item[(a)] $\EE(\U,\V)=0$.
\smallskip

\item[(b)] For any object $B\in \B$, there exist two $\EE$-triangles
\begin{align*}
 V^B\rightarrow  U^B\rightarrow B{\dashrightarrow},\quad
B\rightarrow  V_B\rightarrow  U_B{\dashrightarrow}
\end{align*}
satisfying $U_B,U^B\in \U$ and $V_B,V^B\in \V$.
\end{itemize}
\end{defn}

\begin{rem}
For an extension closed subcategory $\M$, we say a pair of subcategories $(\X,\Y)$ is a cotorsion pair in $\M$ if $\X\subseteq \M$, $\Y\subseteq \M$ and $(\X,\Y)$ is a cotorsion pair in the extriangulated category $(\M,\EE|_{\M},\mathfrak{s}|_{\M})$.
\end{rem}

By the definition of a cotorsion pair, we can immediately conclude the following result.

\begin{lem}\label{L1}
Let $(\U,\V)$ be a cotorsion pair in $\B$.
\begin{itemize}
\item[\rm (a)] $\V=\U^{\bot_1}:=\{ X\in \B \text{ }|\text{ } \EE(\U,X)=0\}$.
\item[\rm (b)] $\U={^{\bot_1}}\V:=\{ Y\in \B \text{ }|\text{ } \EE(Y,\V)=0\}$.
\item[\rm (c)] $\U$ and $\V$ are closed under extensions and direct sums.
\item[\rm (d)] $\mathcal I\subseteq \V$ and $\mathcal P\subseteq \U$.
\end{itemize}
\end{lem}

\proof (a) Since $\EE(\U,\V)=0$, we have $\V\subseteq\U^{\bot_1}$.
Conversely, for any $M\in\U^{\bot_1}$, since $(\U,\V)$ is a cotorsion pair,
there exists an $\EE$-triangle
$$M\overset{x}{\longrightarrow}V_M\overset{y}{\longrightarrow}U_M\overset{\delta}{\dashrightarrow}$$
where $U_M\in\U$ and $V_M\in\V$.
Since $M\in\U^{\bot_1}$, we have $\EE(\U,M)=0$.
It follows that $\delta=0$ and then $x$ is a section.
This shows that $M$ is a direct summand of $V_M$.
Since $\V$ is closed under direct summands, we have $M\in\V$.

Dually we can show (b).

(c) By (a) and (b), $\U$ and $\V$ are closed under direct sums.
Now we prove that $\U$ is closed under extensions. Similarly, we can show that
$\V$ is closed under extensions.

Assume that
$$A\overset{f}{\longrightarrow}B\overset{g}{\longrightarrow}C\dashrightarrow$$
is an $\EE$-triangle where $A,C\in\U$. By \cite[Proposition 3.11]{NP}, we have the following
exact sequence:
$$\EE(C,\V)\to\EE(B,\V)\to\EE(A,\V).$$
Since $\U={^{\bot_1}}\V$ by (a) and $A,C\in\U$, we have $\EE(A,\V)=0$ and $\EE(C,\V)=0$.
It follows that $\EE(B,\V)=0$ which implies $B\in{^{\bot_1}}\V=\U$.

(d) By \cite[Proposition 3.24]{NP} and its dual,
we have $\EE(\mathcal P,\V)=0$ and $\EE(\U,\mathcal I)=0$.
By (a) and (b), we obtain that $\mathcal P\subseteq \U$ and  $\mathcal I\subseteq \V$. \qed

\begin{rem}
There is no simple relation between torsion and cotorsion pairs in the exact categories as in the triangulated categories. For example,
assume that $\mathcal A$ is an exact category with enough injectives $\mathcal I$, then $(\mathcal  A,\mathcal I)$ is a cotorsion pair in $\mathcal A$, but for any non-zero subcategory $\C\subseteq \mathcal A$, $(\C,\mathcal I)$ is not a torsion pair.
\end{rem}

We introduce a kind of important cotorsion pairs, which can be  regarded as a generalization of co-$t$-structures.

\begin{defn}
A cotorsion pair $(\U,\V)$ is called \emph{hereditary} if it satisfies the following conditions:
\begin{itemize}
\item[\rm (a)] For any $\EE$-triangle $A\to B\to C\dashrightarrow$, $B,C\in \U$ implies $A\in \U$.
\item[\rm (b)] For any $\EE$-triangle $A\to B\to C\dashrightarrow$, $A,B\in \V$ implies $C\in \V$.
\end{itemize}
\end{defn}

When $\B$ has enough projectives and enough injectives, we can define higher extensions $\EE^i(-,-), i\geq 1$ ($\EE^1:=\EE$) of the bifunctor $\EE$ (see \cite[Section 5.1]{LN} for details). Any $\EE$-triangle $A\to B\to C\dashrightarrow$ induces the following long exact sequences:
\begin{align*}
\EE(X,A)\to \EE(X,B)\to \EE(X,C)\to \EE^2(X,A)\to \EE^2(X,B)\to \EE^2(X,C)\to\cdots;\\
\EE(C,X)\to \EE(B,X)\to \EE(A,X)\to \EE^2(C,X)\to \EE^2(B,X)\to \EE^2(A,X)\to\cdots~~
\end{align*}
where $X$ is an arbitrary object in $\B$. Moreover, we have $\EE^i(\mathcal P,-)=0$ and $\EE^i(-,\mathcal I)=0$ for any positive integer $i$. We have the following proposition.

\begin{prop}\label{L2}
When $\B$ has enough projectives and enough injectives,  for any cotorsion pair $(\U,\V)$, the following conditions are equivalent:
\begin{itemize}
\item[\rm (i)] $\EE^2(\U,\V)=0$.
\item[\rm (ii)] For any $\EE$-triangle $A\to B\to C\dashrightarrow$, $B,C\in \U$ implies $A\in \U$.
\item[\rm (iii)] For any $\EE$-triangle $A\to B\to C\dashrightarrow$, $A,B\in \V$ implies $C\in \V$.
\end{itemize}
\end{prop}

\begin{proof}
We show that (i)$\Leftrightarrow$(ii), (i)$\Leftrightarrow$(iii) is by dual.

(i)$\Rightarrow$(ii): For any $\EE$-triangle $A\to B\to C\dashrightarrow$ with $B,C\in \U$, we have an exact sequence:
$$0=\EE(B,V)\to\EE(A,V)\to \EE^2(C,V)=0$$
for any $V\in \V$, which implies that $\EE(A,V)=0$. By Lemma \ref{L1}, we obtain that $A\in \U$.

(ii)$\Rightarrow$(i): Let $U\in \U$ and $V\in \V$. Since $\B$ has enough projectives, $U$ admits an $\EE$-triangle $U'\to P\to U\dashrightarrow$ with $P\in \mathcal P$. By Lemma \ref{L1}, $P\in \U$, hence $U'\in \U$. Then we have an exact sequence:
$$0=\EE(U',V)\to\EE^2(U,V)\to \EE^2(P,V)=0$$
which implies $\EE^2(U,V)=0$.
\end{proof}

Let $(\U_1,\V_1)$, $(\U_2,\V_2)$ be two cotorsion pairs. By Lemma \ref{L1}, we can find that
$$\U_1\subseteq \U_2\Leftrightarrow \V_2\subseteq \V_1\Leftrightarrow\EE(\U_1,\V_2)=0.$$

\begin{defn}
A pair of cotorsion pairs $((\U_1,\V_1),(\U_2,\V_2))$ is called a \emph{twin cotorsion pair} if $\U_1\subseteq \U_2$.
\end{defn}


In the rest of this paper, let $((\X,\V),(\U,\Y))$ be a twin cotorsion pair. For convenience, we put $\W:=\X\cap\Y$.

Denote by $[\W](A,B)$ the subgroup of $\Hom_{\B}(A,B)$ consisting of the morphisms $f$ factoring through objects in $\W$. We denote by $\B/[\W]$ (or $\overline \B$ for short) the category which has the same objects as $\B$, and
$$\Hom_{\oB}(A,B)=\Hom_{\B}(A,B)/[\W](A,B)$$
for any $A,B\in \B$. For any morphism $f\in \Hom_{\B}(A,B)$, we denote its image in $\Hom_{\oB}(A,B)$ by $\overline f$.

\begin{lem}\label{iso}
If $\X\cap \V=\U\cap \Y$, then $\Hom_{\oB}(\U,\Y)=0$ and $\Hom_{\oB}(\X,\V)=0$.
\end{lem}

\begin{proof}
If $\X\cap \V=\U\cap \Y$, then $\X\cap \V=\W=\U\cap \Y$. Let $u:U\to Y$ be a morphism such that $U\in \U$ and $Y\in \Y$. $Y$ admits an $\EE$-triangle $Y'\to U'\xrightarrow{y} Y\dashrightarrow$ where $U'\in \Y\cap \U=\W$ and $Y'\in \Y$. Since $\EE(U,Y')=0$, there is a morphism $u':U\to U'$ such that $yu'=u$. Hence $\overline u=0$,  which implies $\Hom_{\oB}(\U,\Y)=0$.

Dually, we can show that $\Hom_{\oB}(\X,\V)=0$.
\end{proof}

Let
$$\s_L=\{B\in \B \text{ }|\text{ } \exists \text{ }\EE\text{-triangle } Y\to X\to B\dashrightarrow \text{ with }X\in \X\text{ and }Y\in \Y\},$$
$$\s_R=\{B\in \B \text{ }|\text{ } \exists \text{ }\EE\text{-triangle } B\to Y'\to X'\dashrightarrow \text{ with }X'\in \X\text{ and }Y'\in \Y\}.$$

\begin{lem}\label{summands}
If $\X\cap \V=\U\cap \Y$, then $\s_L$ and $\s_R$ are closed under direct summands.
\end{lem}

\begin{proof}
We show that $\s_R$ is closed under direct summands,  and the other half is by dual.\\
Let $A_1\oplus A_2\in \s_R$. Then it admits an $\EE$-triangle $A_1\oplus A_2\to Y_0\to X_0\dashrightarrow$ with $Y_0\in \Y$ and $X_0\in \X$. For $i\in\{1,2\}$, $A_i$ admits an $\EE$-triangle $A_i \to Y_i\to U_i\dashrightarrow$ with $Y_i\in \Y$ and $U_i\in \U$. Hence we have the following commutative diagrams
$$\xymatrix@C=0.8cm@R0.7cm{
A_1\oplus A_2 \ar[r] \ar@{=}[d] &Y_1\oplus Y_2 \ar[r] \ar[d] &U_1\oplus U_2 \ar[d]^{\alpha} \ar@{-->}[r] &\\
A_1\oplus A_2 \ar[r] \ar@{=}[d] &Y_0 \ar[r] \ar[d] &X_0 \ar[d]^{\beta} \ar@{-->}[r] &\\
A_1\oplus A_2 \ar[r]  &Y_1\oplus Y_2 \ar[r]  &U_1\oplus U_2  \ar@{-->}[r] &,\\
}\quad  \xymatrix@C=0.8cm@R0.7cm{
A_1\oplus A_2 \ar[r] \ar[d]_0 &Y_1\oplus Y_2 \ar[r] \ar[d] &U_1\oplus U_2 \ar[d]^{1-\beta\alpha} \ar@{-->}[r] &\\
A_1\oplus A_2 \ar[r]  &Y_1\oplus Y_2 \ar[r]  &U_1\oplus U_2  \ar@{-->}[r] &.
}
$$
 By \cite[Corollary 3.5]{NP}, the right diagram implies that $1-\beta\alpha$ factors through $Y_1\oplus Y_2$. This implies $U_1\oplus U_2$ is a direct summand of $X_0\oplus Y_1\oplus Y_2$. By Lemma \ref{iso}, any morphism from $U_1\oplus U_2$ to $Y_1\oplus Y_2$ factors through $\W$, then any indecomposable object of $U_1\oplus U_2$ is isomorphic to an object in $\X$. Hence $U_i\in \X$, which implies $A_i\in \s_R$. Then $\s_R$ is closed under direct summands.
\end{proof}

\begin{rem}
By the proof of Lemma \ref{iso}, if $\X\cap \V=\U\cap \Y$, then
\begin{itemize}
\item[(a)] any object $X\in \X$ admits an $\EE$-triangle $X\to W_X\to X'\dashrightarrow$ with $W_X\in \W$ and $X'\in \X$;
\item[(b)] any object $Y\in \Y$ admits an $\EE$-triangle $Y'\to W_Y\to Y\dashrightarrow$ with $W_Y\in \W$ and $Y'\in \Y$.
\end{itemize}
Hence $\X,\Y\subseteq \s_L\cap\s_R$.
\end{rem}

We also have the following observation.

\begin{lem}\label{cap}
Let $\s$ be an extension closed subcategory of $\B$ such that $(\X,\Y)$ is a cotorsion pair in $\s$.  Then $\X\cap \V=\Y\cap \U$. Moreover, $\s$ is closed under direct summands.
\end{lem}

\begin{proof}
By definition we have $\X\cap \V\supseteq \X\cap \Y$. Let $Z\in \X\cap \V$. Since $\X\subseteq \s$, $(\X,\Y)$ is a cotorsion pair in $\s$, $Z$ admits an $\EE$-triangle $Z\to W\to X\dashrightarrow$ with $X\in \X$ and $W\in \X\cap  \Y=\W$. This $\EE$-triangle splits, which implies $Z$ is a direct summand of $W\in \W$, hence $Z\in \W$. Dually we can show that $\Y\cap \U=\W$.

Since $\s\subseteq \s_R\cap \s_L$, by Lemma \ref{summands}, we get that $\s$ is closed under direct summands.
\end{proof}

\begin{defn}\label{thick}
A subcategory $\M$ of $\B$ is called \emph{a thick subcategory} in $\B$  provided that it is closed under direct summands and for any $\EE$-triangle
$$A\to B\to C\dashrightarrow$$
in $\M$, if any two objects of $A,B$ and $C$ belong to $\M$, then so is the third one.
\end{defn}

\begin{lem}\label{L3}
If $(\X,\Y)$ is a cotorsion pair in a thick subcategory $\s$, then $\s_R=\s_L=\s$.
\end{lem}

\begin{proof}
By definition, we have $\s\subseteq \s_R\cap \s_L$. But on the other hand, since $\s$ is thick and $\X\subseteq \s,\Y\subseteq \s$, by definition we have $\s_R\subseteq \s$ and $\s_L\subseteq \s$. Hence $\s_R=\s_L=\s$.
\end{proof}

\begin{rem}\label{Hov}
If $((\X,\V),(\U,\Y))$ satisfies the condition $\s_R=\s_L$, then it is called a \emph{Hovey twin cotorsion pair} (see \cite[Definition 5.1]{NP} for more details). Hence if $(\X,\Y)$ satisfies the condition in Lemma \ref{L3}, $((\X,\V),(\U,\Y))$ is a Hovey twin cotorsion pair. Note that if $((\X,\V),(\U,\Y))$ is Hovey, we always have $\X\cap\V=\U\cap \Y$ by \cite[Remark 4.17 , Remark 5.2]{NP}.
\end{rem}

\section{Localizations and quotient subcategories}

In an $\EE$-triangle $A\xrightarrow{~x~}B\xrightarrow{~y~}C\dashrightarrow$ , $x$ is called an {\it inflation} and $y$ is called a {\it deflation}. From now on, we also assume $\B$ satisfies condition (WIC) (\cite[Condition 5.8]{NP}):

\begin{itemize}
\item If we have a deflation $h: A\xrightarrow{~f~} B\xrightarrow{~g~} C$, then $g$ is also a deflation.
\item If we have an inflation $h: A\xrightarrow{~f~} B\xrightarrow{~g~} C$, then $f$ is also an inflation.
\end{itemize}

Note that this condition automatically holds on triangulated categories and Krull-Schmidt exact categories.

Under this condition,  we can show the following lemma.

\begin{lem}\label{deflation}
Let $\C$ be a subcategory of $\B$ which is closed under direct summands.
\begin{itemize}
\item[{\rm (a)}] If $f:C\to B$ is a right $\C$-approximation of $B$ and a deflation, then there is a deflation $f_1:C_1\to B$ which is a minimal right $\C$-approximation.
\item[{\rm (b)}] If $f':B\to C'$ is a left $\C$-approximation of $B$ and an inflation, then there is an inflation $f_1':B\to C_1'$ which is a minimal left $\C$-approximation.
\end{itemize}
\end{lem}

\begin{proof}
We only need to prove (a). Dually we can show (b).

If $f:C\to B$ is a deflation, since $\B$ is Krull-Schmidt, there exists a decomposition $C=C_1\oplus C_2$ such that $f=\svech{f_1}{0}\colon C_1\oplus C_2\to B$ where $f_1$ is right minimal. Since $\svech{f_1}{0}=f_1\circ \svech{1}{0}$, by condition (WIC), we get that $f_1$ is also a deflation. Let $C_0$ be any object in $\C$ and $g\in \Hom_\B(C_0,B)$. Since $f$ is a right $\C$-approximation of $B$, there is a morphism $\svecv{h_1}{h_2}:C_0\to C_1\oplus C_2$ such that $g=f\circ \svecv{h_1}{h_2}=f_1h_1$. Hence $f_1$ is a right $\C$-approximation of $B$.
\end{proof}

\subsection{Extriangulated quotient categories}
In the rest of the paper, let $\Z=\U\cap \V$. Since $\U$ and $\V$ are closed under extensions, $\Z$ is also closed under extensions, hence it is an extriangulated subcategory of $\B$.
Note that $\EE(\W,\Z)=0=\EE(\Z,\W)$, by definition, $\W$ is a subcategory of projective-injective objects in $\Z$. Then by \cite[Proposition 3.30]{NP},  $\Z/[\W]$ has an extriangulated structure induced by $\Z$.


Moreover, we have the following proposition.

\begin{prop}\label{Fro}
If $(\X,\V)$ and $(\U,\Y)$ are hereditary cotorsion pairs, then $\Z$ is a Frobenius subcategory in which $\W$ is the subcategory of enough projective-injective objects, which implies that $\Z/[\W]$ is a triangulated category.
\end{prop}

\begin{proof}
According to \cite[Corollary 7.4]{NP}, we only need to show that $\W$ is the subcategory of enough projective-injective objects. Let $Z\in \Z$ be any object. It admits an $\EE$-triangle $V\to X\to Z\dashrightarrow$ where $X\in \X$ and $V\in \V$. Since $Z,X\in \U$, by definition, we have $V\in \U\cap \V=\Z$. Moreover, since $V,Z\in \V$, we have $X\in \V$, hence $X\in \X\cap \V=\W$ by Lemma \ref{cap}. Dually we can show that $Z$ admits an $\EE$-triangle $Z\to Y\to U\dashrightarrow$ where $Y\in \W$ and $U\in Z$. Thus $\Z$ is a Frobenius subcategory in which $\W$ is the subcategory of projective-injective objects, which implies that $\Z/[\W]$ is a triangulated category by \cite[Theorem 3.13]{ZZ}.
\end{proof}

If $(\X,\V)$ and $(\U,\Y)$ are hereditary cotorsion pairs, then the triangulated category structure on $\Z/[\W]$ is the following:

the suspension functor
$$\langle1\rangle: A\mapsto A\langle1\rangle,\ a\mapsto a\langle1\rangle$$
and distinguished triangles
$$A\xrightarrow{\overline f} B\xrightarrow{\overline g} C\xrightarrow{\overline h} A\langle 1\rangle$$
are given by the following commutative diagram:
$$\xymatrix@C=0.8cm@R0.7cm{
A \ar[r]^f \ar@{=}[d] &B \ar[r]^g \ar[d] &C \ar[d]^h \ar@{-->}[r] &\\
A \ar[r] \ar[d]_a &W_A \ar[r] \ar[d] &A\langle 1\rangle \ar[d]^{a\langle 1\rangle} \ar@{-->}[r] &\\
D \ar[r] &W_D \ar[r] &D\langle 1\rangle \ar@{-->}[r] &
}
$$
with $W_A,W_D\in \W$. Here $\xymatrix@C=0.7cm@R0.7cm{A \ar[r]^f  &B \ar[r]^g  &C \ar@{-->}[r] &}$ is an arbitrary  $\EE$-triangle in $\Z$ and $a:A\to D$ is an arbitrary morphism in $\Z$.

\subsection{Hovey twin cotorsion pairs}

In \cite[Section 5]{NP}, Nakaoka-Palu showed that if we have a Hovey twin cotorsion pair $((\X,\V),(\U,\Y))$, then we can get an equivalence between $\Z/[\W]$ and a localization with respect to $((\X,\V),(\U,\Y))$. By Lemma \ref{L3}, we can find that the torsion pairs in \cite[Theorem 1.1]{IYa2} induces a Hovey twin cotorsion pair, and this theorem becomes a special case of the results in \cite[Section 5]{NP}.

A question is that given a Hovey twin cotorsion pair $((\X,\V),(\U,\Y))$, can we find an extension closed subcategory $\s$ in which $(\X,\Y)$ is a cotorsion pair. We can answer this question by the following proposition.

\begin{prop}\label{Hot}
Let $((\X,\V),(\U,\Y))$ be a Hovey twin cotorsion pair (which means $\s_L=\s_R$). Then $\s:=\s_L(=\s_R)$ is a thick subcategory in which $(\X,\Y)$ is a cotorsion pair.
\end{prop}

\begin{proof}
We only need to show that $\s$ is a thick subcategory, then by definition we know that $(\X,\Y)$ is a cotorsion pair in $\s$.

(1) By Lemma \ref{summands}, $\s$ is closed under direct summands.

Let $A\xrightarrow{x} B\xrightarrow{y} C\dashrightarrow$ be an $\EE$-triangle.

(2) If $A,C\in \s$, we show that $B\in \s$. $A$ admits an $\EE$-triangle $A\to Y_A\to X_A\dashrightarrow$ with $Y_A\in \Y$ and $X_A\in \X$, and $C$ admits an $\EE$-triangle $Y^C\to X^C\to C\dashrightarrow$ with $Y^C\in \Y$ and $X^C\in \X$, then we have the following commutative diagrams
$$\xymatrix@C=0.8cm@R0.7cm{
A \ar[r] \ar[d] &B \ar[r] \ar[d] &C \ar@{=}[d] \ar@{-->}[r] &\\
Y_A \ar[r] \ar[d] &D \ar[r] \ar[d] &C \ar@{-->}[r] &,\\
X_A \ar@{=}[r] \ar@{-->}[d] &X_A \ar@{-->}[d]\\
&&\\
}\quad \xymatrix@C=0.8cm@R0.7cm{
&Y^C \ar@{=}[r] \ar[d] &Y^C \ar[d]\\
Y_A \ar[r] \ar@{=}[d] &Y_A\oplus X^C \ar[r] \ar[d] &X^C \ar@{-->}[r] \ar[d] &\\
Y_A \ar[r] &D \ar[r] \ar@{-->}[d] &C \ar@{-->}[r] \ar@{-->}[d] &.\\
&&
}
$$
Since $((\X,\V),(\U,\Y))$ is a Hovey twin cotorsion pair, we always have $\X\cap\V=\W=\U\cap\Y$ by Remark \ref{Hov}, then $Y_A$ admits an $\EE$-triangle $Y_1\to W_1\to Y_A\dashrightarrow$ with $Y_1\in \Y$ and $W_1\in \W$  by the proof of Lemma \ref{iso}. We have the following commutative diagram
$$\xymatrix@C=0.8cm@R0.7cm{
Y_1 \ar@{=}[r] \ar[d] &Y_1 \ar[d]\\
Y_2 \ar[r] \ar[d] &W_1\oplus X^C \ar[r] \ar[d] &D \ar@{=}[d] \ar@{-->}[r] &\\
Y^C \ar[r] \ar@{-->}[d] &Y_A\oplus X^C \ar[r] \ar@{-->}[d] &D \ar@{-->}[r] &\\
&&
}
$$
with $Y_2\in \Y$, hence $D\in \s$. Then $D$ admits an $\EE$-triangle $D\to Y_D\to X_D\dashrightarrow$ with $Y_D\in \Y$ and $X_D\in \X$. We have the following commutative diagrams:
$$\xymatrix@C=0.8cm@R0.7cm{
B \ar[r] \ar@{=}[d] &D \ar[r] \ar[d] &X_A \ar@{-->}[r] \ar[d] &\\
B \ar[r] &Y_D \ar[r] \ar[d] &X' \ar@{-->}[r] \ar[d] &\\
&X_D \ar@{=}[r] \ar@{-->}[d] &X_D \ar@{-->}[d]\\
&&
}
$$
with $X'\in \X$, hence $B\in \s$.

(3) If $A,B\in \s$, we show that $C\in \s$. In the previous argument, we know that $C$ admits  a commutative diagram
$$\xymatrix@C=0.8cm@R0.7cm{
A \ar[r] \ar[d] &B \ar[r] \ar[d] &C \ar@{=}[d] \ar@{-->}[r] &\\
Y_A \ar[r] \ar[d] &D \ar[r] \ar[d] &C \ar@{-->}[r] &.\\
X_A \ar@{=}[r] \ar@{-->}[d] &X_A \ar@{-->}[d]\\
&&
}
$$
Since $B\in \s$ and $\s$ is closed under extensions, we get that $D\in \s$. Then $D$ admits an $\EE$-triangle $Y^D\to X^D\to D\dashrightarrow$ with $Y^D\in \Y$ and $X^D\in \X$. We get the following commutative diagram
$$\xymatrix@C=0.8cm@R0.7cm{
Y^D \ar@{=}[r] \ar[d] &Y^D \ar[d]\\
Y' \ar[r] \ar[d] &X^D \ar[r] \ar[d] &C \ar@{=}[d] \ar@{-->}[r] &\\
Y_A \ar[r] \ar@{-->}[d] &D \ar[r] \ar@{-->}[d] &C \ar@{-->}[r] &\\
&&
}
$$
with $Y'\in \Y$, hence $C\in \s$.

(4) Dually we can show that $B,C\in \s$ implies $A\in \s$.

Hence $\s$ is a thick subcategory.
\end{proof}

The following proposition gives a sufficient condition when $((\X,\V),(\U,\Y))$ becomes a Hovey twin cotorsion pair.

\begin{prop}\label{Here}
Let $((\X,\V),(\U,\Y))$ be a twin cotorsion pair. If
\begin{itemize}
\item[\rm (a)] $\X\cap \V=\U\cap \Y$;
\item[\rm (b)] $(\X,\V)$ and $(\U,\Y)$ are hereditary cotorsion pairs,
\end{itemize}
then $((\X,\V),(\U,\Y))$ is a Hovey twin cotorsion pair. Moreover, if $\M$ is an extension closed subcategory in which $(\X,\Y)$ is a cotorsion pair, then $\M=\s_L=\s_R$.
\end{prop}

\begin{proof}
Let $\s=\s_L\cap\s_R$. By the condition (a), we know that $\X,\Y\subseteq \s$. Then by Lemma \ref{L3}  and Remark \ref{Hov}, we only need to show that $\s$ is a thick subcategory.

(1) By Lemma \ref{summands}, $\s$ is closed under direct summands.

Let $A\xrightarrow{x} B\xrightarrow{y} C\dashrightarrow$ be an $\EE$-triangle.

(2) If $A,C\in \s$, we show that $B\in \s$. $A$ admits an $\EE$-triangle $A\to Y_A\to X_A\dashrightarrow$ with $Y_A\in \Y$ and $X_A\in \X$,  and $C$ admits an $\EE$-triangle $Y^C\to X^C\to C\dashrightarrow$ with $Y^C\in \Y$ and $X^C\in \X$, then we have the following commutative diagrams
$$\xymatrix@C=0.8cm@R0.7cm{
A \ar[r] \ar[d] &B \ar[r] \ar[d] &C \ar@{=}[d] \ar@{-->}[r] &\\
Y_A \ar[r] \ar[d] &D \ar[r] \ar[d] &C \ar@{-->}[r] &,\\
X_A \ar@{=}[r] \ar@{-->}[d] &X_A \ar@{-->}[d]\\
&&\\
}\quad \xymatrix@C=0.8cm@R0.7cm{
&Y^C \ar@{=}[r] \ar[d] &Y^C \ar[d]\\
Y_A \ar[r] \ar@{=}[d] &Y_A\oplus X^C \ar[r] \ar[d] &X^C \ar@{-->}[r] \ar[d] &\\
Y_A \ar[r] &D \ar[r] \ar@{-->}[d] &C \ar@{-->}[r] \ar@{-->}[d] &.\\
&&
}
$$
Since  by the proof of Lemma \ref{iso}, $Y_A$ admits an $\EE$-triangle $Y_1\to W_1\to Y_A\dashrightarrow$ with $Y_1\in \Y$ and $W_1\in \W$,  and $X^C$ admits an $\EE$-triangle $X^C\to W_2\to X_2\dashrightarrow$ with $X_2\in \X$ and $W_2\in \W$, we have the following commutative diagrams
$$\xymatrix@C=0.8cm@R0.7cm{
Y_C \ar[r] \ar@{=}[d] &Y_A\oplus X^C \ar[r] \ar[d] &D \ar[d] \ar@{-->}[r] &\\
Y_C \ar[r] &Y_A\oplus W_2 \ar[r] \ar[d] &Y' \ar[d] \ar@{-->}[r] &,\\
&X_2 \ar@{=}[r] \ar@{-->}[d] &X_2 \ar@{-->}[d]\\
&&\\
}\quad \xymatrix@C=0.8cm@R0.7cm{
Y_1 \ar@{=}[r] \ar[d] &Y_1 \ar[d]\\
Y_2 \ar[r] \ar[d] &W_1\oplus X^C \ar[r] \ar[d] &D \ar@{=}[d] \ar@{-->}[r] &\\
Y_C \ar[r] \ar@{-->}[d] &Y_A\oplus X^C \ar[r] \ar@{-->}[d] &D \ar@{-->}[r] &\\
&&
}
$$
with $Y_2\in \Y$.  Since $(\U,\Y)$ is hereditary, by definition we have $Y'\in \Y$. Then we have the following commutative diagrams:
$$\xymatrix@C=0.8cm@R0.7cm{
B \ar[r] \ar@{=}[d] &D \ar[r] \ar[d] &X_A \ar@{-->}[r] \ar[d] &\\
B \ar[r] &Y' \ar[r] \ar[d] &X' \ar@{-->}[r] \ar[d] &,\\
&X_2 \ar@{=}[r] \ar@{-->}[d] &X_2 \ar@{-->}[d]\\
&&\\
}\hspace{1.2cm}\xymatrix@C=0.8cm@R0.7cm{
Y_2 \ar@{=}[r] \ar[d] &Y_2 \ar[d]\\
X'' \ar[r] \ar[d] &W_1\oplus X^C \ar[r] \ar[d] &X_A \ar@{=}[d] \ar@{-->}[r] &\\
B \ar[r] \ar@{-->}[d] &D \ar[r] \ar@{-->}[d] &X_A \ar@{-->}[r] &\\
&&
}
$$
with $X'\in \X$.  Since $(\X,\V)$ is hereditary, by definition we have $X''\in \X$. Hence $B\in \s$ and $\s$ is closed under extensions.

(3) If $A,B\in \s$, we show that $C\in \s$. In the previous argument, we know that $C$ admits {\red a} commutative diagram
$$\xymatrix@C=0.8cm@R0.7cm{
A \ar[r] \ar[d] &B \ar[r] \ar[d] &C \ar@{=}[d] \ar@{-->}[r] &\\
Y_A \ar[r] \ar[d] &D \ar[r] \ar[d] &C \ar@{-->}[r] &.\\
X_A \ar@{=}[r] \ar@{-->}[d] &X_A \ar@{-->}[d]\\
&&
}
$$
Since $B\in \s$, and $\s$ is closed under extensions, we get that $D\in \s$.  Then $D$ admits an $\EE$-triangle $D\to Y_D\to X_D\dashrightarrow$ with $Y_D\in \Y$ and $X_D\in \X$. We have the following commutative diagram
$$\xymatrix@C=0.8cm@R0.7cm{
Y_A \ar[r] \ar@{=}[d] &D \ar[r]^d \ar[d]^y &C \ar@{-->}[r] \ar[d]^c &\\
Y_A \ar[r] &Y_D \ar[r]^{y_0} \ar[d] &Y_0 \ar@{-->}[r] \ar[d] &.\\
&X_D \ar@{=}[r] \ar@{-->}[d] &X_D \ar@{-->}[d]\\
&&
}
$$
 Since $(\U,\Y)$ is hereditary, by definition we have $Y_0\in \Y$. By \cite[Proposition 1.20]{LN}, we can choose a morphism $y_0':Y_D\to Y_0$ to make an $\EE$-triangle $D\xrightarrow{\svecv{-d}{y}} C\oplus Y_D\xrightarrow{\svech{c}{y_0'}} Y_0\dashrightarrow$.

Since $\s$ is closed under extensions and direct summands, we get that $C\in \s$.

(4) Dually we can show that $B,C\in \s$ implies $A\in \s$.

Hence $\s$ is a thick subcategory. Moreover, we have $\s=\s_R=\s_L$.

Let $\M$ be an extension closed subcategory in which $(\X,\Y)$ is a cotorsion pair. Then by definition $\M\subseteq \s$. Let $S$ be any object in $\s$. $S$ admits an $\EE$-triangle $Y^S\to X^S\to S\dashrightarrow$ with $Y^S\in \Y$ and $X^S\in \X${\red .} $X^S$ admits an $\EE$-triangle $X^S\to W_0\to X_0\dashrightarrow$ with $W_0\in \W$ and $X_0\in \X$. We have the following commutative diagram
$$\xymatrix@C=0.8cm@R0.7cm{
Y^S \ar[r] \ar@{=}[d] &X^S \ar[r] \ar[d] &S \ar@{-->}[r] \ar[d] &\\
Y^S \ar[r] &W_0 \ar[r] \ar[d] &Y \ar@{-->}[r]. \ar[d] &\\
&X_0 \ar@{=}[r] \ar@{-->}[d] &X_0 \ar@{-->}[d]\\
&&
}
$$
 Since $(\U,\Y)$ is hereditary, by definition we have $Y\in \Y$. By the similar argument as above, we can get an $\EE$-triangle $X^S \to S\oplus W_0\to Y\dashrightarrow$. Since $\M$ is closed under extensions, and by Lemma \ref{cap}, $\M$ is closed under direct summands, we have $S\in \M$. This means $\s\subseteq \M$. Then we get $\M=\s$.
\end{proof}

By proving Proposition \ref{Hot} and Proposition \ref{Here}, we have the  following question: if $((\X,\V),(\U,\Y))$ is not Hovey, can we still get a similar equivalence between the ideal quotient $\Z/[\W]$ and a localization with respect to an extension closed subcategory? On the other hand, we know that the ideal quotients are extriangulated in Theorem \ref{1.1} and Theorem \ref{1.2}, and at the same time the localizations have triangulated category structures.  A natural question is that do these equivalences preserve extriangulated category structures?

In the rest of the paper,  we assume that there is an extension closed subcategory $\s$ in which $(\X,\Y)$ is a cotorsion pair.

We first show that under certain condition, there exists an equivalence between the ideal quotient $\Z/[\W]$ and the localization with respect to $\s$. Then we show that when $\s$ is thick, the equivalence from $\Z/[\W]$ to the localization has better properties.

\subsection{A localization of $\B$ realized by $\Z/[\W]$}


We first establish a functor $G$ from $\B$ to $\Z/[\W]$. For any object $B$ of $\B$, we fix two $\EE$-triangles
\begin{align*}
\xymatrix@C=0.6cm@R0.5cm{B\ar[r]^{v_B} &V_B\ar[r] &X_B \ar@{-->}[r] &},\quad
\xymatrix@C=0.6cm@R0.5cm{Y^B\ar[r] &Z_B\ar[r]^{z_B} &V_B \ar@{-->}[r] &,}
\end{align*}
where $X_B\in \X$, $Y^B\in \Y$, $v_B$ is a minimal left $\V$-approximation and $z_B$ is a minimal right $\U$-approximation. We have $Z_B\in \Z$.  Note that by Lemma \ref{deflation}, we can take these minimal approximations.

Let $f:B\to C$ be any morphism in $\B$. Then we have the following commutative diagrams:
$$\xymatrix@C=0.8cm@R0.7cm{
B \ar[r]^{v_B} \ar[d]^f &V_B \ar[r] \ar[d]^{v_f} &X_B \ar[d] \ar@{-->}[r] &\\
C \ar[r]^{v_C} &V_C \ar[r] &X_C \ar@{-->}[r] &,\\
}\quad
\xymatrix@C=0.8cm@R0.7cm{
Y^B \ar[r] \ar[d] &Z_B \ar[r]^{z_B} \ar[d]^{z_f} &V_B \ar[d]^{v_f} \ar@{-->}[r] &\\
Y^C \ar[r] &Z_C  \ar[r]^{z_C} &V_C \ar@{-->}[r] &.
}
$$

\begin{lem}\label{unique}
If we have $v_f':V_B\to V_C$ and $z_f':Z_B\to Z_C$ such that the following diagrams commute:
$$\xymatrix@C=0.8cm@R0.7cm{
B \ar[r]^{v_B} \ar[d]^f &V_B \ar[r] \ar[d]^{v_f'} &X_B \ar[d] \ar@{-->}[r] &\\
C \ar[r]^{v_C} &V_C \ar[r] &X_C \ar@{-->}[r] &,\\
} \quad
\xymatrix@C=0.8cm@R0.7cm{
Y^B \ar[r] \ar[d] &Z_B \ar[r]^{z_B} \ar[d]^{z_f'} &V_B \ar[d]^{v_f'} \ar@{-->}[r] &\\
Y^C \ar[r] &Z_C  \ar[r]^{z_C} &V_C \ar@{-->}[r] &
}
$$
then $\overline v_f=\overline v_f'$ and $\overline z_f=\overline z_f'$.
\end{lem}

\begin{proof}
If we have the commutative diagrams above, then $v_f-v_f':V_B\to V_C$ factors through $X_B$, hence by  Lemma \ref{cap} and Lemma~\ref{iso} factors through $\W$. Then $z_C\circ(z_f-z_f')$ factors through an object $W\in \W$. Let $z_C\circ(z_f-z_f')=w_2\circ w_1$ with $w_1:Z_B\to W$ and $w_2:W\to V_C$. Then there is a morphism $w_3:W\to Z_C$ such that $w_2=z_C\circ w_3$. Thus $z_C\circ\left((z_f-z_f')-w_3w_1\right)=0$, then $(z_f-z_f')-w_3w_1$ factors through $Y^C$, hence by Lemma~\ref{iso} factors through $\W$. Then $\overline z_f=\overline z_f'$.
\end{proof}

%
%
Now we can define a functor $G$ from $\B$ to $\Z/[\W]$, acting as follows:
$$G(B)=Z_B, \quad G(f)=\overline z_f.$$

\begin{rem}\label{rem1}
By the construction of $G$, we have $G(v_B)=G(z_B)=\overline 1_{Z_B}$ and $G(\s)=0$. Moreover, $G$ is additive.
\end{rem}

\begin{defn}\label{lo}
Denote by $\R$ the following class of morphisms:
$$\{f:B\to C \mid \exists \text{ } \EE\text{-triangle} \xymatrix@C=0.5cm@R0.5cm{B\ar[r]^{f} &C\ar[r] &S \ar@{-->}[r]&} \text{ with } S\in \s\}.$$
Denote by $\B/\s$ the Gabriel-Zisman localization of $\B$ with respect to $\R$
(see \cite[Section I.2]{GZ} or \cite[Section 2.2]{K2} for more details of such localization).
\end{defn}

In this localization, any morphism $f\in \R$ becomes invertible. For any morphism $g$, we denote its image in $\B/\s$ by $\underline g$.


We will show the following theorem.

\begin{thm}\label{main5}
Assume $\s$ is closed under taking cones, which means in any $\EE$-triangle $A\to B\to C\dashrightarrow$, $A,B\in \s$ implies that $C\in \s$. Then the Gabriel-Zisman localization $\B/\s$ is equivalent to $\Z/[\W]$.
\end{thm}

We first show an important property of functor $G$.

\begin{prop}\label{im}
$G(f)$ is an isomorphism for any morphism $f:B\to C$ in $\R$.
\end{prop}

\begin{proof}
Morphism $f$ admits an $\EE$-triangle $\xymatrix@C=0.6cm@R0.7cm{B\ar[r]^{f} &C\ar[r] &S \ar@{-->}[r] &}$ with $S\in
\s$. Then we have the following commutative diagram
$$\xymatrix@C=0.8cm@R0.7cm{
B \ar[r]^{v_B} \ar[d]_f &V_B \ar[r]^{x_B} \ar[d]^{d_1} &X_B \ar@{=}[d] \ar@{-->}[r] &\\
C \ar[r]^{c'} \ar[d] &C' \ar[d]^{d'_2} \ar[r]^{ x'} &X_B \ar@{-->}[r] &.\\
S \ar@{=}[r]  \ar@{-->}[d] &S \ar@{-->}[d]\\
&&
}
$$
Since $V_C\in \V$, we have $\EE(X_B,V_C)=0$, then we get the following commutative diagram
$$\xymatrix@C=0.8cm@R=0.7cm{
B \ar[r]^{v_B} \ar[d]_f &V_B \ar[r]^{x_B} \ar[d]^{d_1} &X_B \ar@{=}[d] \ar@{-->}[r] &\\
C \ar[r]^{c'} \ar[d]_{v_C} &C' \ar[d]^-{\svecv{d_2}{x'}} \ar[r]^{x'} &X_B \ar@{=}[d] \ar@{-->}[r] &\\
V_C \ar[r]^-{\svecv{1}{0}} \ar[d] &V_C\oplus X_B \ar[d] \ar[r] &X_B \ar@{-->}[r] &\\
X_C \ar@{=}[r] \ar@{-->}[d] &X_C \ar@{-->}[d]\\
&&
}
$$
such that $\overline {d_2d_1}=\overline {v_f}$ by the proof of Lemma \ref{unique}. Then we have a commutative diagram
$$\xymatrix@C=0.8cm@R0.7cm{
V_B \ar[r]^{d_1} \ar@{=}[d] &C' \ar[r] \ar[d]^-{\svecv{d_2}{ x'}} &S \ar[d] \ar@{-->}[r] &\\
V_B \ar[r]_-{\svecv{d_2d_1}{x_B}} &V_C\oplus X_B \ar[r] \ar[d] &S_1 \ar[d] \ar@{-->}[r] &\\
&X_C \ar@{=}[r]  \ar@{-->}[d] &X_C \ar@{-->}[d]\\
&&
}$$
with $S_1\in \s$. Since $X_B$ admits an $\EE$-triangle $\xymatrix@C=0.6cm@R0.6cm{X_B\ar[r]^{w} &W\ar[r] &X_1\ar@{-->}[r] &}$ where $W\in \W$ and $X_1\in \X$, we have the following commutative diagram
$$\xymatrix@C=0.8cm@R0.7cm{
V_B \ar[r]^-{\svecv{d_2d_1}{x_B}} \ar@{=}[d] &V_C\oplus X_B \ar[r] \ar[d]^-{\left(\begin{smallmatrix}
1& 0\\
0& w
\end{smallmatrix}\right)} &S_1 \ar[d] \ar@{-->}[r] &\\
V_B \ar[r] &V_C\oplus W \ar[r] \ar[d] &S_2 \ar[d] \ar@{-->}[r] &\\
&X_1 \ar@{=}[r]  \ar@{-->}[d] &X_1 \ar@{-->}[d]\\
&&
}
$$
with $S_2\in \s$. $S_2$ admits an $\EE$-triangle $Y_2\to X_2\to S_2\dashrightarrow$ with $Y_2\in \Y$ and $X_2\in \X$, then we have the following commutative diagram
$$\xymatrix@C=0.8cm@R0.7cm{
&&Y_2 \ar@{=}[r] \ar[d] &Y_2 \ar[d]\\
V_B \ar@{=}[d] \ar[rr]^-{\svecv{1_{V_B}}{0}} &&V_B\oplus X_2 \ar[r] \ar[d]^-{\left(\begin{smallmatrix}
d_2d_1& *\\
*&*
\end{smallmatrix}\right)} &X_2 \ar@{-->}[r] \ar[d] &\\
V_B \ar[rr]_-{\left(\begin{smallmatrix}
d_2d_1& 0\\
0& wx_B
\end{smallmatrix}\right)} &&V_C\oplus W \ar[r] \ar@{-->}[d] &S_2  \ar@{-->}[d] \ar@{-->}[r] &.\\
&&&
}$$
Since $Y_2,V_C\oplus W\in \V$, we get that $X_2\in \X\cap\V=\W$. Now we obtain the following commutative diagrams
$$\xymatrix@C=0.8cm@R0.7cm{
Y^B \ar@{=}[r] \ar[d] &Y^B \ar[d]\\
Y_3 \ar[r] \ar[d] &Z_B\oplus X_2 \ar[rr]^-{\left(\begin{smallmatrix}
 d_2d_1z_B& *\\
*&*
\end{smallmatrix}\right)} \ar[d]^-{\left(\begin{smallmatrix}
z_B& 0\\
0& 1_{X_2}
\end{smallmatrix}\right)} &&V_C\oplus W \ar@{=}[d] \ar@{-->}[r] &\\
Y_2 \ar[r] \ar@{-->}[d] &V_B\oplus X_2 \ar[rr]_-{\left(\begin{smallmatrix}
d_2d_1& *\\
*&*
\end{smallmatrix}\right)} \ar@{-->}[d] &&V_C\oplus W \ar@{-->}[r] &,\\
&&\\
}\quad \xymatrix@C=0.8cm@R0.7cm{
&Y^C \ar@{=}[rr] \ar[d] &&Y^C \ar[d]\\
Y_3 \ar[r] \ar@{=}[d] &Z_B\oplus X_2\oplus Y^C \ar[rr]^-{\left(\begin{smallmatrix}
z& * &*\\
*& * &*
\end{smallmatrix}\right)} \ar[d]^-{\left(\begin{smallmatrix}
1_{Z_B}& 0 &0\\
0& 1_{X_2} &0
\end{smallmatrix}\right)} &&Z_C\oplus W \ar@{-->}[r] \ar[d]^-{\left(\begin{smallmatrix}
z_C& 0\\
0& 1_W
\end{smallmatrix}\right)} &\\
Y_3 \ar[r] &Z_B\oplus X_2 \ar[rr]_-{\left(\begin{smallmatrix}
d_2d_1z_B& *\\
*&*
\end{smallmatrix}\right)} \ar@{-->}[d] &&V_C\oplus W \ar@{-->}[d] \ar@{-->}[r] &\\
&&&
}
$$
with $Y_3\in \Y$. Since  $\EE(Z_C\oplus W,Y_3)=0$, we have  the following commutative diagram:
$$
\xymatrix@C=0.8cm@R0.7cm{
Y_3 \ar[r] \ar@{=}[d] &Z_B\oplus X_2\oplus Y^C \ar[rr]^-{\left(\begin{smallmatrix}
z& * &*\\
*& * &*
\end{smallmatrix}\right)} \ar@{.>}[d]^-{\left(\begin{smallmatrix}
z& * &*\\
*& * &*\\
*& * &*
\end{smallmatrix}\right)}_-{\simeq} &&Z_C\oplus W \ar@{-->}[r]^-0 \ar@{=}[d] &\\
Y_3 \ar[r]_-{\left(\begin{smallmatrix}
0& 0 &1
\end{smallmatrix}\right)} &Z_C\oplus W\oplus Y_3 \ar[rr]_-{\left(\begin{smallmatrix}
1& 0 &0\\
0& 1& 0
\end{smallmatrix}\right)} &&Z_C\oplus W  \ar@{-->}[r]^-0 &.
}
$$
By Lemma \ref{iso}, $\Hom_{\oB}(Z_B,Y_3)=0$,  hence $\overline z$ is a section and $Z_B$ is a direct summand of $Z_C$ in $\oB$. On the other hand, if an indecomposable object is a direct summand of both $Y^C$ and $Z_C$, then it has to  lie in $\W$. Hence $Y^C$ is a direct summand of $Y_3$ in $\oB$. Since $Z_B\oplus Y^C\simeq Z_C\oplus Y_3$ in $\oB$, we obtain that $\overline z$ is an isomorphism. Since  $\overline {d_2d_1 z_B}=\overline {v_fz_B}=\overline {z_Cz}$, by the proof of Lemma \ref{unique}, we get that $\overline z=\overline z_f$. Hence $\overline z_f$ is an isomorphism.
\end{proof}

By Proposition \ref{im} and the universal property of the localization functor $L_\R:\B\to \B/\s$, we obtain the following commutative diagram:
$$\xymatrix@C=0.8cm@R0.7cm{
\B\ar[rr]^-{G} \ar[dr]_{L_\R} &&\Z/[\W].\\
&\B/\s \ar@{.>}[ur]_{H}
}
$$

The following lemma is useful. The proof is an analogue of \cite[Lemma 3.5]{BM}, so we omit it.

\begin{lem}\label{BML}
\begin{itemize}
\item[\rm (1)] Let $X$ be any object in $\B$ and $S\in \s$. Then $X\oplus S\xrightarrow{\svech{1_X}{0}}X$ is invertible in $\B/\s$,  and its inverse is $X\xrightarrow{\svecv{1_X}{0}} X\oplus S$.
\item[\rm (2)] Let $B,C$ be any objects in $\B$ and $f,f'\in \Hom_\B(B,C)$. If $\overline f'=0$, then  $\underline{f+ f'}=\underline f$ in $\B/\s$.
\end{itemize}
\end{lem}

In the rest of this subsection, we assume that $\s$ is closed under taking cones.

\begin{cor}\label{BMLcor}
For any $\EE$-triangle $$\xymatrix@C=0.6cm@R=0.6cm{Y\ar[r] &U\ar[r]^{u} &A\ar@{-->}[r] &}$$
with $Y\in \Y$, $U\in \U$, $u$ is invertible in $\B/\s$.
\end{cor}

\begin{proof}
Since $U$ admits an $\EE$-triangle $U\xrightarrow{w} W\to U'\dashrightarrow$ with $W\in \W$ and $U'\in \U$, we have the following commutative diagram
$$\xymatrix@C=0.8cm@R0.7cm{
Y\ar[r] \ar@{=}[d] &U \ar[r]^u \ar[d]^{w} &A \ar[d]^a \ar@{-->}[r] &\\
Y \ar[r]  &W \ar[r]^s \ar[d] &S \ar[d] \ar@{-->}[r] &\\
&U' \ar@{=}[r]  \ar@{-->}[d] &U' \ar@{-->}[d]\\
&&
}
$$
with $S\in \s$,  since $\s$ is closed under taking cones. By \cite[Proposition 1.20]{LN}, we can choose a morphism $s':W\to S$ to make an $\EE$-triangle $U\xrightarrow{\svecv{-u}{w}} A\oplus W\xrightarrow{\svech{a}{s'}} S\dashrightarrow$. Hence $U\xrightarrow{\svecv{-u}{w}} A\oplus W$ is invertible in $\B/\s$. By Lemma \ref{BML}, $-u=\svech{1}{0}\circ \svecv{-u}{w}:U\to A$ is invertible in $\B/\s$, so is $u$.
\end{proof}

\begin{rem}
Let $f:B\to C$ be any morphism in $\R$. In the following commutative diagrams
$$\xymatrix@C=0.8cm@R0.7cm{
B \ar[r]^{v_B} \ar[d]^f &V_B \ar[r] \ar[d]^{v_f} &X_B \ar[d] \ar@{-->}[r] &\\
C \ar[r]^{v_C} &V_C \ar[r] &X_C \ar@{-->}[r] &,\\
}\quad
\xymatrix@C=0.8cm@R0.7cm{
Y^B \ar[r] \ar[d] &Z_B \ar[r]^{z_B} \ar[d]^{z_f} &V_B \ar[d]^{v_f} \ar@{-->}[r] &\\
Y^C \ar[r] &Z_C  \ar[r]^{z_C} &V_C \ar@{-->}[r] &
}
$$
$\underline v_f$ is invertible. Moreover, since $\underline z_B$ and $\underline z_C$ are invertible by Corollary \ref{BMLcor}, $\underline z_f$ is also invertible in $\B/\s$.
\end{rem}

\begin{lem}\label{lem}
Let $f:B\to C$ be any morphism in $\R$. Then we have the following commutative diagram in $\B/\s$
$$\xymatrix@C=0.8cm@R0.7cm{
C \ar[r]^{\underline v_C} \ar[d]^{\underline f^{-1}} &V_C \ar[r]^{\underline z_C^{-1}} \ar[d]^{ \underline v_f^{-1}} &Z_C \ar[d]^{\underline z}\\
B \ar[r]_{\underline v_B} &V_B \ar[r]_{\underline z_B^{-1}} &Z_B
}
$$
where $z:Z_C\to Z_B$ is a morphism in $\Z$ such that $\overline z=G(f)^{-1}$.
\end{lem}

\begin{proof}
By Proposition \ref{im}, $G(f)=\overline z_f$ is an isomorphism, let $\overline z=G(f)^{-1}$. Then $1_{Z_C}-z_fz$ factors through $\W$. So by Lemma \ref{BML}, we have $\underline 1_{Z_C}=\underline {z_fz}$ in $\B/\s$.
By the following commutative diagram
$$\xymatrix@C=0.8cm@R0.7cm{
B \ar[r]^{v_B} \ar[d]_f &V_B \ar[d]^{v_f} &Z_B \ar[l]_{z_B} \ar[d]^{z_f}\\
C \ar[r]_{v_C} &V_C &Z_C \ar[l]_{z_C}
}
$$
we have $v_fz_Bz=z_Cz_fz$. By applying $L_\R$ to this equation, we get $\underline v_f\underline z_B\underline z=\underline z_C$, then $\underline z\underline z_C^{-1}=\underline z_B^{-1}\underline v_f^{-1} $. Hence  we obtain the desired commutative diagram.
\end{proof}

We now give the proof of Theorem \ref{main5}.

\begin{proof}[Proof of Theorem~\ref{main5}.]
We show that $H$ is an equivalence. Since $G|_{\Z}$ is identical on the objects, we know that $H$ is dense.

We show that $H$ is faithful. Let $\alpha:B\to C$ be any morphism in $\B/\s$. It has the form $B\xrightarrow{\beta_0} D_1\xrightarrow{\beta_1} \cdot\cdot\cdot \xrightarrow{\beta_{n-1}} D_n\xrightarrow{\beta_n} C$ where $\beta_i$ is a morphism $\underline {f_i}$ or a morphism $\underline {g_i}^{-1}$ with $g_i\in \R$. We have a commutative diagram
$$\xymatrix@C=0.8cm@R0.7cm{
B \ar[d]_{\underline z_B^{-1}\underline v_B} \ar[r]^{\beta_0} &D_1 \ar[r]^{\beta_1} \ar[d] &\cdot \cdot \cdot \ar@{}[d]^{\cdots}_{\cdots} \ar[r]^{\beta_{n-1}} &D_n\ar[r]^{\beta_n} \ar[d] &C \ar[d]^{\underline z_C^{-1}\underline v_C}\\
Z_B \ar[r]^{\underline z_0} &Z_1 \ar[r]^{\underline z_1} &\cdot \cdot \cdot \ar[r]^{\underline z_{n-1}} &Z_n\ar[r]^{\underline z_n} &Z_C
}
$$
where $Z_i\in \Z$ and $z_i$ are morphisms in $\Z$ by Lemma \ref{lem}. Denote $z_n z_{n-1}\cdot\cdot\cdot z_1z_0$ by $\zeta$, we have $\alpha=\underline v_C^{-1}\underline z_C\underline \zeta\underline z_B^{-1}\underline v_B$. If there exists a morphism $\alpha':B\to C$ in $\B/\s$ such that $H(\alpha)=H(\alpha')$,  then we also have $\alpha'=\underline v_C^{-1}\underline z_C\underline \zeta'\underline z_B^{-1}\underline v_B$ with some $\zeta'\in \Hom_\B(Z_B,Z_C)$. Since $H(\underline v_C)=H(\underline z_C)=\overline 1_{Z_C}$ and $H(\underline v_B)=H(\underline z_B)=\overline 1_{Z_B}$ by Remark \ref{rem1}, we can get that $\overline {\zeta}=H(\underline {\zeta})=H(\underline {\zeta'})=\overline {\zeta'}$. Hence $\zeta-\zeta'$ factors through $\W$, which implies $\underline {\zeta}=\underline {\zeta'}$. Thus $\alpha=\alpha'$.

Finally we show that $H$ is full. Let $\gamma:H(B)\to H(C)$ be any morphism. By {\red a} similar argument as above, we can get that $\gamma=\overline z$ where $z$ is a morphism in $\Z$. Since we have the following commutative diagram in $\B/\s$:
$$\xymatrix@C=0.8cm@R0.7cm{
B \ar[r]^{\underline v_B} \ar[d]_{\alpha} &V_B \ar[d]^{\underline {z_Czz_B^{-1}}} &Z_B \ar[l]_-{\underline z_B} \ar[d]^{\underline z} \ar@{=}[r] &H(B)\\
C \ar[r]_{\underline v_C} &V_C &Z_C \ar[l]^-{\underline z_C} \ar@{=}[r] &H(C)
}
$$
we have $H(\alpha)=H(\underline v_C)^{-1}H(\underline z_C)H(\underline z)H(\underline z_B)^{-1}H(\underline v_B)$, then $H(\alpha)=H(\underline z)=\overline z=\gamma$, hence $H$ is full.
\end{proof}


Since we have the following commutative diagram
$$\xymatrix{
\Z \ar[r]^{i} \ar[dr]_{\pi} &\B \ar[r]^-{L_\R} &\B/\s\\
 &\Z/[\W] \ar@{.>}[ur]_{F}
}
$$
where $i$ is the embedding functor and $\pi$ is the canonical quotient functor, we can get the following proposition.

\begin{prop}\label{quasi}
$F$ is a quasi-inverse of $H$.
\end{prop}

\begin{proof}
By the proof of Theorem \ref{main5} we can obtain $HF=\Id_{\Z/[\W]}$. Let $\alpha:B\to C$ be any morphism in $\B/\s$. Then we have the following commutative diagram
$$\xymatrix{
FH(B) \ar[rr]^{\underline v_B^{-1} \underline z_B}_{\simeq} \ar[d]_{FH(\alpha)} &&B \ar[d]^{\alpha}\\
FH(C) \ar[rr]_{\underline v_C^{-1} \underline z_C}^{\simeq} &&C
}
$$
which implies $FH\cong \Id_{\B/\s}$.
\end{proof}

\begin{rem}
Assume that $\B$ has enough injectives. Let $\M$ be a contravariantly finite subcategory which contains all the injective objects and is closed under extensions, direct summands and taking cones. Then $(\M,\mathcal I)$ is a cotorsion pair in $\M$. Moreover, we have a twin cotorsion pair $((\M,\M^{\bot_1}),(\B,\mathcal I))$ in $\B$. By Theorem \ref{main5}, $\M^{\bot_1}/[\mathcal I]\simeq \B/\M$.

Dually, assume that $\B$ has enough projectives. Let $\N$ be a covariantly finite subcategory which  contains all the projective objects and is closed under extensions, direct summands and taking cocones. Then $(\mathcal P,\N)$ is a cotorsion pair in $\N$. Moreover, we have a twin cotorsion pair $((\mathcal P,\B),({^{\bot_1}}\N,\N))$ in $\B$ and ${^{\bot_1}}\N/[\mathcal P]\simeq \B/\N$.

In particular, if $\B$ is a triangulated category and $\M$ is a contravariantly (resp. covariantly) finite thick subcategory, then the Verdier quotient $\B/\M$ is a triangle equivalent to $\M^{\bot_1}$ (resp. ${^{\bot_1}\M}$), which is a thick subcategory of $\B$.
\end{rem}

\begin{rem}
There are many examples of extriangulated subcategories which have enough injectives. For example, any right triangulated category with a right semi-equivalence is an extriangulated category with enough injectives, see \cite[Corollary 3.13]{T}.
\end{rem}

\subsection{Other localizations}

%

 In this subsection, we assume that $\s$ is closed under taking cones. We first show a useful proposition.

\begin{prop}\label{Sigma}
For any $\EE$-triangle $Y\to A\xrightarrow{f} B\dashrightarrow$ with $Y\in \Y$, $G(f)$ is an isomorphism.
\end{prop}

\begin{proof}
Since $A$ admits an $\EE$-triangle $A\to Y_A\to U_A\dashrightarrow$ with $Y_A\in \Y$ and $U_A\in \U$, we have the following commutative diagram
$$\xymatrix@C=0.8cm@R0.7cm{
Y \ar[r] \ar@{=}[d] &A \ar[r]^f \ar[d] &B \ar[d] \ar@{-->}[r] &\\
Y \ar[r] &Y_A \ar[r] \ar[d] &S \ar[d] \ar@{-->}[r] &.\\
&U_A \ar@{=}[r] \ar@{-->}[d] &U_A \ar@{-->}[d]\\
&&
}
$$
 Since $\s$ is closed under taking cones, we have $S\in \s$. By \cite[Proposition 1.20]{LN}, there is an $\EE$-triangle $A\xrightarrow{\svecv{-f}{*}} B\oplus Y_A\to S\dashrightarrow$. By Proposition \ref{im}, we get that $G(f)$ is an isomorphism.
\end{proof}

Denote by $\mathbb{W}_1$ the following class of morphisms:
$$\{f:A\to B \mid \exists \text{ } \EE \text{-triangle } \xymatrix@C=0.5cm@R0.5cm{A\ar[r]^{f} &B\ar[r] &X \ar@{-->}[r]&} \text{ with } X\in \X \}.$$
Denote by $\mathbb{W}_2$ the following class of morphisms:
$$\{g:C\to D \mid \exists \text{ } \EE \text{-triangle } \xymatrix@C=0.5cm@R0.5cm{Y\ar[r] &C\ar[r]^g &D \ar@{-->}[r]&} \text{ with } Y\in \Y \}.$$
Denote by $\mathbb{W}$ the following class of morphisms:
$$\{h=g\circ f \mid f\in \mathbb{W}_1 \text{ and }g\in \mathbb{W}_2\}.$$
Denote by $\B[\mathbb{W}^{-1}]$ the Gabriel-Zisman localization of $\B$ with respect to $\mathbb{W}$,  and by $L:\B\to \B[\mathbb{W}^{-1}]$ the localization functor.

\begin{rem}
If $\s$ is a thick subcategory, then by Lemma \ref{L3} we have $\s_R=\s_L$ and $((\X,\V),(\U,\Y))$ is a Hovey twin cotorsion pair. $\B[\mathbb{W}^{-1}]$ becomes the localization discussed in \cite[Section 5]{NP}.
\end{rem}

Since by  Propositions \ref{im} and \ref{Sigma}, morphisms in $\mathbb{W}_1$ and $\mathbb{W}_2$ become invertible in $\B/\s$, there exists a unique functor $F_1:\B[\mathbb{W}^{-1}]\to \B/\s$ such that $F_1L=L_\R$. On the other hand, we have the following proposition.

\begin{prop}\label{RW}
Morphisms in $\R$ are invertible in $\B[\mathbb{W}^{-1}]$.
\end{prop}

\begin{proof}
Let $f:A\to B$ be a morphism in $\R$. Then it admits an $\EE$-triangle $A\xrightarrow{f} B\to S\dashrightarrow$ with $S\in\s$. Since $S$ admits an $\EE$-triangle $Y\to X\to S\dashrightarrow$ with $Y\in \Y$ and $X\in \X$, we have the following commutative diagram:
$$\xymatrix@C=0.8cm@R0.7cm{
&Y \ar@{=}[r] \ar[d] &Y \ar[d]\\
A \ar[r]^{d_1} \ar@{=}[d] &D \ar[r] \ar[d]^{d_2} &X \ar[d] \ar@{-->}[r]&\\
A \ar[r]_f &B \ar[r] \ar@{-->}[d] &S \ar@{-->}[r] \ar@{-->}[d] &\\
&&
}
$$
 which implies $f\in \mathbb{W}$. Hence $\R\subseteq \mathbb{W}$ and morphisms in $\R$ are invertible in $\B[\mathbb{W}^{-1}]$.
\end{proof}

By this proposition, there exists a unique functor $F_2: \B/\s\to \B[\mathbb{W}^{-1}]$ such that $F_2L_\R=L$. Hence $L=F_2F_1L$ and $L_\R=F_1F_2L_\R$. By the universal property of the localization functors, we have $F_2F_1=\Id_{\B[\mathbb{W}^{-1}]}$ and $F_2F_1=\Id_{\B/\s}$. This means $\B/\s$ is  equivalent to $\B[\mathbb{W}^{-1}]$.\\

We claim that when $\s$ is closed under taking cones, $((\X,\V),(\U,\Y))$ is a generalized Hovey twin cotorsion pair  in the sense of \cite[Definition 3.9]{O}. By \cite[Definition 3.9 and Theorem 3.10]{O}, we only need to check the following fact:
$$\text{For an }\EE\text{-triangle }Y\to U\xrightarrow{u} B\dashrightarrow \text{ with }Y\in \Y, U\in \U,G(u)\text{ is an isomorphism in }\Z/[\W].$$
This is just followed by Proposition \ref{Sigma}.\\

Denote by $\mathbb{V}_1$ the following class of morphisms:
$$\{f:A\to V \mid \exists \text{ } \EE \text{-triangle } \xymatrix@C=0.5cm@R0.5cm{A\ar[r]^{f} &V\ar[r] &X \ar@{-->}[r]&} \text{ with } V\in \V,X\in \X \}.$$
Denote by $\mathbb{V}_2$ the following class of morphisms:
$$\{g:U\to B \mid \exists \text{ } \EE \text{-triangle } \xymatrix@C=0.5cm@R0.5cm{Y\ar[r] &U\ar[r]^g &B \ar@{-->}[r]&} \text{ with } Y\in \Y,U\in \U \}.$$
Denote by $\mathbb{V}$ the following class of morphisms:
$$\{h=g\circ f \mid f\in \mathbb{V}_1 \text{ and }g\in \mathbb{V}_2\}.$$
Denote by $\B[\mathbb{V}^{-1}]$ the Gabriel-Zisman localization of $\B$ with respect to $\mathbb{V}${\red .} Ogawa \cite[Theorem 3.10]{O} showed that there is an equivalence: $\Phi:\Z/[\W]\to \B[\mathbb{V}^{-1}]$. Then
$$\Phi \circ H: \B/\s\to \B[\mathbb{V}^{-1}] \quad \text{and} \quad \Phi \circ H\circ F_1:\B[\mathbb{W}^{-1}]\to \B[\mathbb{V}^{-1}]$$
are equivalences between localizations.
\medskip

In summary, we have the following proposition.

\begin{prop}
Let $\s$ be an extension closed subcategory which is also closed under taking cones. Let $((\X,\V),(\U,\Y))$ be a twin cotorsion pair such that $(\X,\Y)$ is a cotorsion pair in $\s$. Then we have the following commutative diagram:
$$\xymatrix{
&\B \ar[dl]_-L \ar[d]^{L_\R}  \ar[dr]^G \\
\B[\mathbb{W}^{-1}]\ar[r]_-{F_1}^{\simeq} &\B/\s \ar[r]_-H^{\simeq} &\Z/[\W] \ar[r]_{\Phi}^{\simeq} &\B[\mathbb{V}^{-1}].
}$$
\end{prop}

\begin{rem}
According to \cite[Theorem 3.14]{O}, if $((\X,\V),(\U,\Y))$ is a Hovey twin cotorsion pair, then we always have an equivalence between $\B[\mathbb{W}^{-1}]$ and $\B[\mathbb{V}^{-1}]$. Note that in general we do not have such equivalence.
\end{rem}

In the following example, the twin cotorsion pair is not Hovey.

\begin{exm}\label{example}
Let $Q\colon 1\to 2\to 3\to 4$ be the quiver of type $A_4$ and $\B:=\mathrm{D}^b(kQ)$ the bounded derived category of $kQ$ whose Auslander-Reiten quiver is the following:
$$\xymatrix@C=0.3cm@R0.2cm{
\cdot\cdot\cdot \ar[dr] &&\circ \ar[dr] &&\circ \ar[dr] &&\circ \ar[dr] &&\circ \ar[dr] &&\circ \ar[dr] &&\circ \ar[dr] &&\bullet \ar[dr] &&\circ \ar[dr] &&\circ \ar[dr]\\
&\circ \ar[ur] \ar[dr] &&\circ \ar[ur] \ar[dr] &&\circ \ar[ur] \ar[dr] &&\circ \ar[ur]  \ar[dr] &&\circ\ar[ur] \ar[dr]  &&\circ \ar[ur] \ar[dr] &&\bullet \ar[ur]  \ar[dr] &&\bullet \ar[ur] \ar[dr] &&\circ \ar[ur] \ar[dr] &&\cdot\cdot\cdot\\
\cdot\cdot\cdot \ar[dr] \ar[ur] &&\circ \ar[ur] \ar[dr] &&\circ \ar[ur] \ar[dr] &&\circ \ar[ur] \ar[dr]  &&\circ \ar[ur] \ar[dr] &&\circ \ar[ur] \ar[dr] &&\bullet  \ar[ur] \ar[dr] &&\bullet \ar[ur] \ar[dr] &&\circ \ar[ur] \ar[dr] &&\circ \ar[ur] \ar[dr]\\
&\circ \ar[ur] &&\circ \ar[ur] &&\circ \ar[ur] &&\circ \ar[ur] &&\circ \ar[ur] &&\bullet \ar[ur] &&\bullet \ar[ur] &&\circ \ar[ur] &&\circ \ar[ur] &&\cdot\cdot\cdot
}
$$
Let $\X$ be the subcategory whose indecomposable objects are marked by $\bullet$ in the diagram. Let $\Y$ be the subcategory whose indecomposable objects are marked by $\clubsuit$ in the diagram ($\clubsuit$ will continue to appear on the right side of the diagram):
$$\xymatrix@C=0.3cm@R0.2cm{
\cdot\cdot\cdot \ar[dr] &&\circ \ar[dr] &&\circ \ar[dr] &&\circ \ar[dr] &&\circ \ar[dr] &&\circ \ar[dr] &&\circ \ar[dr] &&\clubsuit \ar[dr] &&\clubsuit \ar[dr] &&\clubsuit \ar[dr]\\
&\circ \ar[ur] \ar[dr] &&\circ \ar[ur] \ar[dr] &&\circ \ar[ur] \ar[dr] &&\circ \ar[ur]  \ar[dr] &&\circ\ar[ur] \ar[dr]  &&\circ \ar[ur] \ar[dr] &&\circ \ar[ur]  \ar[dr] &&\clubsuit \ar[ur] \ar[dr] &&\clubsuit \ar[ur] \ar[dr] &&\cdot\cdot\cdot\\
\cdot\cdot\cdot \ar[dr] \ar[ur] &&\circ \ar[ur] \ar[dr] &&\circ \ar[ur] \ar[dr] &&\circ \ar[ur] \ar[dr]  &&\circ \ar[ur] \ar[dr] &&\circ \ar[ur] \ar[dr] &&\circ  \ar[ur] \ar[dr] &&\clubsuit \ar[ur] \ar[dr] &&\clubsuit\ar[ur] \ar[dr] &&\clubsuit \ar[ur] \ar[dr]\\
&\circ \ar[ur] &&\circ \ar[ur] &&\circ \ar[ur] &&\circ \ar[ur] &&\circ \ar[ur] &&\circ \ar[ur] &&\clubsuit \ar[ur] &&\clubsuit \ar[ur] &&\clubsuit \ar[ur] &&\cdot\cdot\cdot
}
$$
Let $\mathcal S$ be the smallest full subcategory of $\B$ that contains both $\X$ and $\Y$. Then $\s$ is an extension closed subcategory such that $\s[1]\subseteq \s$ (hence $\s$ is closed under taking cones), and $(\X,\Y)$ is a cotorsion pair in $\s$. $\V=\X^{\bot_1}$ is the subcategory whose indecomposable objects are marked by $\spadesuit$ in the diagram: ($\spadesuit$ will continue to appear on the both sides of the diagram):
$$\xymatrix@C=0.3cm@R0.2cm{
\cdot\cdot\cdot \ar[dr] &&\spadesuit \ar[dr] &&\spadesuit \ar[dr] &&\circ \ar[dr] &&\circ \ar[dr] &&\circ \ar[dr] &&\circ \ar[dr] &&\spadesuit \ar[dr] &&\spadesuit \ar[dr] &&\spadesuit \ar[dr]\\
&\spadesuit \ar[ur] \ar[dr] &&\spadesuit \ar[ur] \ar[dr] &&\spadesuit \ar[ur] \ar[dr] &&\circ \ar[ur]  \ar[dr] &&\circ\ar[ur] \ar[dr]  &&\circ \ar[ur] \ar[dr] &&\circ \ar[ur]  \ar[dr] &&\spadesuit \ar[ur] \ar[dr] &&\spadesuit \ar[ur] \ar[dr] &&\cdot\cdot\cdot\\
\cdot\cdot\cdot \ar[dr] \ar[ur] &&\spadesuit \ar[ur] \ar[dr] &&\spadesuit \ar[ur] \ar[dr] &&\spadesuit \ar[ur] \ar[dr]  &&\circ \ar[ur] \ar[dr] &&\circ \ar[ur] \ar[dr] &&\circ  \ar[ur] \ar[dr] &&\spadesuit \ar[ur] \ar[dr] &&\spadesuit \ar[ur] \ar[dr] &&\spadesuit \ar[ur] \ar[dr]\\
&\spadesuit \ar[ur] &&\spadesuit \ar[ur] &&\spadesuit \ar[ur] &&\spadesuit \ar[ur] &&\circ \ar[ur] &&\circ \ar[ur] &&\spadesuit \ar[ur] &&\spadesuit \ar[ur] &&\spadesuit \ar[ur] &&\cdot\cdot\cdot
}
$$
$\U={^{\bot_1}}\Y$ is the subcategory whose indecomposable objects are marked by $\heartsuit$ in the diagram: ($\heartsuit$ will continue to appear on the left side of the diagram):
$$\xymatrix@C=0.3cm@R0.2cm{
\cdot\cdot\cdot \ar[dr] &&\heartsuit\ar[dr] &&\heartsuit \ar[dr] &&\heartsuit \ar[dr] &&\heartsuit \ar[dr] &&\heartsuit \ar[dr] &&\heartsuit \ar[dr] &&\heartsuit \ar[dr] &&\circ \ar[dr] &&\circ \ar[dr]\\
&\heartsuit \ar[ur] \ar[dr] &&\heartsuit \ar[ur] \ar[dr] &&\heartsuit \ar[ur] \ar[dr] &&\heartsuit \ar[ur]  \ar[dr] &&\heartsuit\ar[ur] \ar[dr]  &&\heartsuit \ar[ur] \ar[dr] &&\heartsuit \ar[ur]  \ar[dr] &&\heartsuit \ar[ur] \ar[dr] &&\circ \ar[ur] \ar[dr] &&\cdot\cdot\cdot\\
\cdot\cdot\cdot \ar[dr] \ar[ur] &&\heartsuit \ar[ur] \ar[dr] &&\heartsuit \ar[ur] \ar[dr] &&\heartsuit \ar[ur] \ar[dr]  &&\heartsuit \ar[ur] \ar[dr] &&\heartsuit \ar[ur] \ar[dr] &&\heartsuit  \ar[ur] \ar[dr] &&\heartsuit \ar[ur] \ar[dr] &&\circ\ar[ur] \ar[dr] &&\circ \ar[ur] \ar[dr]\\
&\heartsuit \ar[ur] &&\heartsuit\ar[ur] &&\heartsuit \ar[ur] &&\heartsuit \ar[ur] &&\heartsuit \ar[ur] &&\heartsuit \ar[ur] &&\heartsuit \ar[ur] &&\circ \ar[ur] &&\circ \ar[ur] &&\cdot\cdot\cdot
}
$$
The indecomposable objects in $\Z/[\W]$ are marked by $\maltese$ in the diagram: ($\maltese$ will continue to appear on the left side of the diagram):
$$\xymatrix@C=0.3cm@R0.2cm{
\cdot\cdot\cdot \ar[dr] &&\maltese \ar[dr] &&\maltese \ar[dr] &&\bigstar \ar[dr] &&\bigstar \ar[dr] &&\bigstar \ar[dr] &&\bigstar \ar[dr] &&\circ \ar[dr] &&\circ \ar[dr] &&\circ \ar[dr]\\
&\maltese \ar[ur] \ar[dr] &&\maltese \ar[ur] \ar[dr] &&\maltese \ar[ur] \ar[dr] &&\bigstar \ar[ur]  \ar[dr] &&\bigstar\ar[ur] \ar[dr]  &&\bigstar \ar[ur] \ar[dr] &&\circ \ar[ur]  \ar[dr] &&\circ \ar[ur] \ar[dr] &&\circ \ar[ur] \ar[dr] &&\cdot\cdot\cdot\\
\cdot\cdot\cdot \ar[dr] \ar[ur] &&\maltese \ar[ur] \ar[dr] &&\maltese \ar[ur] \ar[dr] &&\maltese \ar[ur] \ar[dr]  &&\bigstar \ar[ur] \ar[dr] &&\bigstar \ar[ur] \ar[dr] &&\circ  \ar[ur] \ar[dr] &&\circ \ar[ur] \ar[dr] &&\circ \ar[ur] \ar[dr] &&\circ \ar[ur] \ar[dr]\\
&\maltese \ar[ur] &&\maltese \ar[ur] &&\maltese \ar[ur] &&\maltese\ar[ur] &&\bigstar \ar[ur] &&\circ \ar[ur] &&\circ \ar[ur] &&\circ \ar[ur] &&\circ \ar[ur] &&\cdot\cdot\cdot
}
$$
Note that the indecomposable objects marked by $\bigstar$ in the diagram become zero objects in $\B/\s$. The twin cotorsion pair in this example is {\bf NOT} a Hovey twin cotorsion pair, since the objects marked by $\bigstar$ lie in $\s_R$, but do not lie in $\s_L$.
\end{exm}

\subsection{Some observations}
Observe that we replace the assumption ``$\s$ is closed under taking cones" by ``$(\U,\Y)$ is a hereditary cotorsion pair", all the results still hold. In fact, we have the following proposition.

\begin{prop}\label{P1}
If $(\U,\Y)$ is a hereditary cotorsion pair, then $\s$ is closed under taking cones. Moreover, $\s=\s_L$.
\end{prop}

\begin{proof}
Let $S_1\to S_2\to S\dashrightarrow$ be an $\EE$-triangle with $S_1,S_2\in \s$. Since $S_1$ admits an $\EE$-triangle $S_1\to Y_1\to X_1\dashrightarrow$ with $Y_1\in \Y$ and $X_1\in \X$, we have the following commutative diagram
$$\xymatrix@C=0.8cm@R0.7cm{
S_1 \ar[r] \ar[d] &S_2 \ar[r]  \ar[d] &S \ar@{-->}[r] \ar@{=}[d] &\\
Y_1 \ar[r] \ar[d] &S_3 \ar[r] \ar[d] &S \ar@{-->}[r]&\\
X_1 \ar@{=}[r] \ar@{-->}[d] &X_1 \ar@{-->}[d]\\
&&
}
$$
with $S_3\in\s$. Since $S_3$ admits an $\EE$-triangle $S_3\to Y_3\to X_3\dashrightarrow$ with $Y_3\in \Y$ and $X_3\in \X$, we have the following commutative diagram
$$\xymatrix@C=0.8cm@R0.7cm{
Y_1 \ar[r] \ar@{=}[d] &S_3 \ar[r] \ar[d] &S \ar@{-->}[r] \ar[d]&\\
Y_1 \ar[r] &Y_3 \ar[d] \ar[r] &Y \ar@{-->}[r] \ar[d]&\\
&X_3 \ar@{=}[r] \ar@{-->}[d] &X_3 \ar@{-->}[d]\\
&&
}
$$
with $Y\in \Y$. Then we can get an $\EE$-triangle $S_3\to S\oplus Y_3\to Y\dashrightarrow$. Since $\s$ is closed under extensions and direct summands, we have $S\in \s$.

By definition, $\s\subseteq\s_L$. Let $A\in \s_L$. Then $A$ admits an $\EE$-triangle $Y^A\to X^A\to A\dashrightarrow$ with $Y^A\in \Y$ and $X^A\in \X$. Since $X^A$ admits an $\EE$-triangle $X^A\to W\to X\dashrightarrow$ with $W\in \W$ and $X\in \X$, we have the following commutative diagram
$$\xymatrix@C=0.8cm@R0.7cm{
Y^A \ar[r] \ar@{=}[d] &X^A \ar[r] \ar[d] &A \ar[d] \ar@{-->}[r]&\\
Y^A \ar[r] &W \ar[r] \ar[d] &Y'  \ar[d] \ar@{-->}[r]&\\
&X \ar@{=}[r] \ar@{-->}[d] &X \ar@{-->}[d]\\
&&
}$$
with $Y'\in \Y$. Then we can get an $\EE$-triangle $X^A\to A\oplus W\to Y'\dashrightarrow$, which implies that $A\in \s$. Hence $\s=\s_L$.
\end{proof}

At the end of this section, we drop the assumption that ``$\s$ is an extension closed subcategory in which $(\X,\Y)$ is a cotorsion pair". We give a sufficient condition when $(\X,\Y)$ is a cotorsion pair in an extension closed subcategory which is closed under taking cones.

We first show a useful lemma.

\begin{lem}\label{extension}
Let $\C$ and $\D$ be two subcategories of $\B$ which are closed under extensions. Denote by $\C*\D$ the following subcategory:
$$\{B\in \B \text{ }|\text{ } \exists \text{ }\EE\text{-triangle } C\to B\to D\dashrightarrow \text{ with }C\in \C\text{ and }D\in \D\}.$$
If $\EE(\C,\D)=0$, then $\C*\D$ is closed under extensions.
\end{lem}

\begin{proof}
For convenience, denote $\C*\D$ by $\K$. Let $K_1\to K\to K_2\dashrightarrow$ be an $\EE$-triangle with $K_1,K_2\in \K$. Since $K_2$ admits an $\EE$-triangle $C_2\to K_2\to D_2\dashrightarrow$ with $C_2\in \C$ and $D_2\in \D$, we get the following commutative diagram
$$\xymatrix@C=0.8cm@R0.7cm{
K_1\ar@{=}[r] \ar[d] &K_1 \ar[d]\\
K' \ar[r] \ar[d] &K \ar[r] \ar[d] &D_2 \ar@{=}[d] \ar@{-->}[r] &\\
C_2 \ar[r] \ar@{-->}[d] &K_2 \ar[r] \ar@{-->}[d] &D_2 \ar@{-->}[r] &.\\
&&
}
$$
Since $K_1$ admits an $\EE$-triangle $C_1\to K_1\to D_1\dashrightarrow$ with $C_1\in \C$, $D_1\in \D$  and $\EE(C_2,D_1)=0$, we get the following commutative diagrams:
$$\xymatrix@C=0.8cm@R0.7cm{
C_1 \ar@{=}[d] \ar[r] &K_1 \ar[r] \ar[d] &D_1 \ar[d] \ar@{-->}[r] &\\
C_1 \ar[r] &K' \ar[r] \ar[d] &D_1\oplus C_2 \ar[d] \ar@{-->}[r] &,\\
&C_2 \ar@{=}[r] \ar@{-->}[d] &C_2 \ar@{-->}[d]\\
&&\\
}\hspace{1.2cm} \xymatrix@C=0.8cm@R0.7cm{
C_1 \ar@{=}[r] \ar[d] &C_1 \ar[d]\\
C_3 \ar[r] \ar[d] &K' \ar[r] \ar[d] &D_1 \ar@{=}[d] \ar@{-->}[r] &\\
C_2 \ar[r] \ar@{-->}[d] &D_1\oplus C_2 \ar[r] \ar@{-->}[d] &D_1 \ar@{-->}[r] &\\
&&
}
$$
with $C_3\in \C$. Now we get the following commutative diagram:
$$\xymatrix@C=0.8cm@R0.7cm{
C_3 \ar[r] \ar@{=}[d] &K' \ar[r] \ar[d] &D_1 \ar[d] \ar@{-->}[r] &\\
C_3 \ar[r] &K \ar[r] \ar[d] &D_3 \ar[d] \ar@{-->}[r] &\\
&D_2 \ar@{=}[r] \ar@{-->}[d] &D_2 \ar@{-->}[d]\\
&&
}
$$
with $D_3\in \D$. Hence $K\in \K$.
\end{proof}

\begin{prop}\label{triHov}
Let $\B$ be an triangulated category with shift functor {\rm [1]}. Assume that the twin cotorsion pair $((\X,\V),(\U,\Y))$ satisfies the following conditions:
\begin{itemize}
\item[\rm (1)] $\X\cap\V=\U\cap\Y$;
\item[\rm (2)] $(\U,\Y)$ is a hereditary cotorsion pair.
\end{itemize}
Then $\M=\s_L$ is the only extension closed subcategory in which $(\X,\Y)$ is a cotorsion pair. Moreover, $\M$ is closed under taking cones.
\end{prop}

\begin{proof}
Since $\B$ is triangulated, we have $\s_R=\X[-1]*\Y$ and $\s_L=\X*\Y[1]$. Since $(\U,\Y)$ is hereditary, we have $\Hom_{\B}(\X,\Y[2])=0$. By Lemma \ref{extension}, $\s_L$ and $\s_R$ are closed under extensions. Hence $\M=\s_L\cap \s_R$ is an extriangulated subcategory of $\B$ in which $(\X,\Y)$ is a cotorsion pair. By Proposition \ref{P1}, we have $\M=\s_L\subseteq \s_R$. Since $\Y[1]\subseteq \Y\subseteq \M$, we have $\M[1]\subseteq \s_R[1]=\s_L=\M$, which implies that $\M$ is closed under taking cones.

If there is another extension closed subcategory $\M_1$ in which $(\X,\Y)$ is a cotorsion pair, then we immediately get $\M_1\subseteq \M$. Any object $M\in \M$ admits an triangle $Y_M\to X_M\to M\to Y_M[1]$ with $Y_M\in \Y$ and $X_M\in \X$, since $Y_M[1]\in \Y$ and $\X,\Y\subseteq\M_1$, we have $M\in \M_1$. Hence $\M_1=\M$.
\end{proof}

\section{Extriangle equivalences}

\begin{defn}\label{exeq}
Let $(\B_1,\EE_1,\mathfrak{s}_1)$ and $(\B_2,\EE_2,\mathfrak{s}_2)$ be two extriangulated categories and $\sigma: \B_1\to \B_2$ be an equivalent functor. $\sigma$ is called an \emph{extriangle equivalence} if the following conditions are satisfied:
\begin{itemize}
\item[(1)] $\sigma$ preserves $\EE$-triangles, which means for any $\EE_1$-triangle $A_1\xrightarrow{x_1} B_1\xrightarrow{y_1} C_1\dashrightarrow$, there exists an $\EE_2$-extension $\delta\in \EE_2(\sigma(C_1),\sigma(A_1))$ such that $\sigma(A_1)\xrightarrow{\sigma(x_1)} \sigma(B_1)\xrightarrow{\sigma(y_1)} \sigma(C_1)\overset{\delta}{\dashrightarrow}$ is an $\EE_2$-triangle.
\item[(2)] For any $\EE_2$-triangle $A_2\xrightarrow{x_2} B_2\xrightarrow{y_2} C_2\dashrightarrow$ in $\B_2$, there exists an $\EE_1$-triangle $A_1\xrightarrow{x_1} B_1\xrightarrow{y_1} C_1\dashrightarrow$ in $\B_1$ admitting an isomorphism of $\EE_2$-triangles
$$\xymatrix@C=0.8cm@R0.7cm{
A_2 \ar[r]^{x_2} \ar[d]^{\simeq} &B_2 \ar[r]^{y_2} \ar[d]^{\simeq} &C_2 \ar@{-->}[r] \ar[d]^{\simeq} &\\
\sigma(A_1) \ar[r]_{\sigma(x_1)} &\sigma(B_1) \ar[r]_{\sigma(y_1)} &\sigma(C_1) \ar@{-->}[r] &.
}$$
\end{itemize}
\end{defn}

\begin{rem}
If $\B_1$ and $\B_2$ are exact categories, then any extriangle equivalence is an exact equivalence (see \cite[Definition 5.1]{B}). If $\B_i$ is a triangulated category with shift functor $[1]_{\B_i}$ ($i=1,2$), then an extriangle equivalence $\sigma:\B_1\to \B_2$ is a ``weak" triangle equivalence in the sense that we do not know if there is a natural isomorphism between $\sigma\circ [1]_{\B_1}$ and $[1]_{\B_2}\circ \sigma$.
\end{rem}

\begin{rem}
In \cite[Definition 2.32]{B-TS}, a functor called $n$-exangulated functor was defined between $n$-exangulated categories. When $n=1$, such functor is called an extriangulated functor (between extriangulated categories, see also \cite[Definition 2.11]{NOS}). According to \cite[Theorem 2.33]{B-TS}, our definition is slightly weak than \cite[Definition 2.32]{B-TS} when $n=1$, since we do not assume the existence of a natural transformation from $\EE_1(-,-)$ to $\EE_2(\sigma^{\op}-,\sigma-)$.
\end{rem}

In this section, we assume that $\B$ has enough enough projectives and enough injectives, and $\s$ is a thick subcategory of $\B$. We show that $F\colon \Z/[\W]\to \B/\s$ is an extriangle equivalence.

Since $((\X,\V),(\U,\Y))$ is a Hovey twin cotorsion pair by Lemma \ref{L3},  by the discussion in Subsection 3.4, we can use the results in \cite[Section 6]{NP}.

Let $f:A\to B$ be any morphism in $\B$. We have the following commutative diagram
$$\xymatrix@C=0.8cm@R0.7cm{
A \ar[r]^{y_A} \ar[d]_f &Y_A \ar[r] \ar[d] &U_A \ar[d]^{u_f} \ar@{-->}[r] &\\
B \ar[r]_{y_B} &Y_B \ar[r] &U_B \ar@{-->}[r] &
}
$$
 where $U_A,U_B\in \U$, $y_A,y_B$ are left minimal $\Y$-approximations. Since $\B/\s$ and $\B[\mathbb{W}^{-1}]$ are equivalent to each other, by the results in \cite[Section 6]{NP}, we can define an auto-equivalence $[1]$ on $\B/\s$ such that $A[1]=U_A$ and $\underline f[1]=\underline u_f$. Moreover, the following commutative diagram
$$\xymatrix@C=0.8cm@R0.7cm{
A_1 \ar[r]^{f_1} \ar[d]_{a_1} &B_1 \ar[r]^{g_1} \ar[d]^{b_1} &C_1 \ar[d]^{c_1} \ar@{-->}[r]&\\
A_2 \ar[r]_{f_2} &B_2 \ar[r]_{g_2} &C_2 \ar@{-->}[r]&\\
}
$$
induces a commutative diagram in $\B/\s$:
$$\xymatrix@C=0.8cm@R0.7cm{
A_1 \ar[r]^{\underline f_1} \ar[d]_{\underline a_1} &B_1 \ar[r]^{\underline g_1} \ar[d]^{\underline b_1} &C_1 \ar[r]^{\alpha_1\;\;\;} \ar[d]^{\underline c_1} &A_1[1] \ar[d]^{\underline a_1[1]}\\
A_2 \ar[r]_{\underline f_2} &B_2 \ar[r]_{\underline g_2} &C_2 \ar[r]_{\alpha_2\;\;\;} & A_2[1].
}
$$
The rows are called \emph{the standard triangles} in $\B/\s$. The \emph{distinguished triangles} are the sequences which are isomorphic to the standard triangles. By \cite[Theorem 6.20]{NP}, $\B/\s$ is a triangulated category with distinguished triangles and the shift functor [1]. Note that any triangulated category can be viewed as an extriangulated category (see \cite[Proposition 3.22]{NP} for details).

\begin{lem}\label{U}
Let $\xymatrix@C=0.6cm@R0.6cm{A\ar[r]^{x} &B\ar[r]^{y} &C\ar@{-->}[r] &}$ be any $\EE$-triangle in $\B$. Then there exists an isomorphism of triangles
$$\xymatrix@C=0.8cm@R0.7cm{
A'\ar[r]^{\underline x'} \ar[d]_{\simeq} &B\ar[r]^{\underline y'} \ar@{=}[d] &C'\ar[d]^{\simeq} \ar[r] &A'[1] \ar[d]^{\simeq}\\
A \ar[r]_{\underline x} &B \ar[r]_{\underline y} &C \ar[r] &A[1]
}
$$
in $\B/\s$ such that the first row admits an $\EE$-triangle $\xymatrix@C=0.6cm@R0.6cm{A' \ar[r]^-{\svecv{x'}{*}} &B\oplus I \ar[r]^-{\svech{y'}{*}} &C'\ar@{-->}[r] &}$ with $A'\in \U$ and $I\in \mathcal I$.
\end{lem}

\begin{proof}
Since $(\U,\Y)$ is a cotorsion pair, $A$ admits an $\EE$-triangle $\xymatrix@C=0.6cm@R0.6cm{Y\ar[r] &A'\ar[r]^a &A\ar@{-->}[r] &}$ with $A'\in \U$ and $Y\in \Y$. $Y$ admits an $\EE$-triangle $\xymatrix@C=0.6cm@R0.6cm{Y \ar[r] &I \ar[r] &S' \ar@{-->}[r] &}$ with $I\in \mathcal I$ and $S'\in \s$. Then we have the following commutative diagrams
$$\xymatrix@C=0.8cm@R0.7cm{
Y \ar[r] \ar[d] &A'\ar[r]^a \ar[d]^-{\svecv{a}{*}} &A \ar@{=}[d] \ar@{-->}[r] &\\
I \ar[r] \ar[d] &A\oplus I \ar[d] \ar[r]_-{\svech{1}{0}} &A \ar@{-->}[r] &,\\
S' \ar@{-->}[d] \ar@{=}[r] &S' \ar@{-->}[d]\\
&&
}\quad \quad \xymatrix@C=0.8cm@R0.7cm{
A' \ar[r]^-{\svecv{a}{*}} \ar@{=}[d] &A\oplus I \ar[r] \ar[d]_-{\left(\begin{smallmatrix}
x& 0\\
0&1
\end{smallmatrix}\right)} &S' \ar[d] \ar@{-->}[r] &\\
A' \ar[r]_-{\svecv{x'}{*}} &B\oplus I \ar[r]^-{\svech{y'}{*}} \ar[d]^-{\svech{y}{0}} &C' \ar[d]^c \ar@{-->}[r] &.\\
&C \ar@{=}[r] \ar@{-->}[d] &C \ar@{-->}[d]\\
&&
}
$$
By Lemma \ref{BML}, $\underline a:A'\to A$ is invertible. $S'$ admits an $\EE$-triangle $S'\to I'\to S''\dashrightarrow$ with $I'\in \mathcal I$ and $S''\in \s$, then we have a commutative diagram
$$\xymatrix@C=0.8cm@R0.7cm{
S' \ar[r] \ar[d] &C'\ar[r]^c \ar[d]^-{\svecv{c}{*}} &C \ar@{=}[d] \ar@{-->}[r] &\\
I' \ar[r] \ar[d] &C\oplus I' \ar[d] \ar[r]_-{\svech{1}{0}} &C \ar@{-->}[r] &.\\
S'' \ar@{-->}[d] \ar@{=}[r] &S'' \ar@{-->}[d]\\
&&
}
$$
By Lemma \ref{BML}, $\underline c:C'\to C$ is invertible. Hence we have the following commutative diagram
$$\xymatrix@C=0.8cm@R0.7cm{
A' \ar[r]^-{\svecv{x'}{*}} \ar[d]_-{\svecv{a}{*}}  &B\oplus I \ar[r]^-{\svech{y'}{*}} \ar@{=}[d] &C' \ar[d]^c \ar@{-->}[r] &\\
A\oplus I \ar[r]^-{\left(\begin{smallmatrix}
x& 0\\
0&1
\end{smallmatrix}\right)} &B\oplus I \ar[r]^-{\svech{y}{0}} &C \ar@{-->}[r] &
}
$$
which induces an isomorphism of triangles
$$\xymatrix@C=0.8cm@R0.7cm{
A'\ar[r]^{\underline x'} \ar[d]^{\simeq}_{\underline a} &B\ar[r]^{\underline y'} \ar@{=}[d] &C'\ar[d]^{\simeq}_{\underline c} \ar[r] &A'[1] \ar[d]^{\simeq}_{\underline a[1]}\\
A \ar[r]_{\underline x} &B \ar[r]_{\underline y} &C \ar[r] &A[1].
}
$$
\end{proof}

Dually, we have the following lemma:

\begin{lem}\label{V}
Let $\xymatrix@C=0.6cm@R0.6cm{A\ar[r]^{x} &B\ar[r]^{y} &C\ar@{-->}[r] &}$ be any $\EE$-triangle in $\B$. Then there exists an isomorphism of triangles
$$\xymatrix@C=0.8cm@R0.7cm{
A\ar[r]^{\underline x} \ar[d]^{\simeq} &B\ar[r]^{\underline y} \ar@{=}[d] &C\ar[d]^{\simeq} \ar[r] &A[1] \ar[d]^{\simeq}\\
A'' \ar[r]_{\underline x''} &B \ar[r]_{\underline y''} &C'' \ar[r] &A''[1]
}
$$
in $\B/\s$ such that the second row admits an $\EE$-triangle $\xymatrix@C=0.6cm@R0.6cm{A'' \ar[r]^-{\svecv{x''}{*}} &B\oplus P \ar[r]^-{\svech{y''}{*}} &C''\ar@{-->}[r] &}$ with $C''\in \V$ and $P\in \mathcal P$.
\end{lem}

Then following proposition shows that $F$ is an extriangle equivalence.

\begin{prop}\label{induce}
Let $\xymatrix@C=0.6cm@R0.6cm{A\ar[r]^{x} &B\ar[r]^{y} &C \ar@{-->}[r] &}$ be any $\EE$-triangle in $\B$. There is an isomorphism between $\EE$-triangles
$$\xymatrix@C=0.8cm@R0.7cm{
A\ar[r]^{\underline x} \ar[d]^{\simeq} &B\ar[r]^{\underline y} \ar[d]^{\simeq} &C\ar[d]^{\simeq}\\
Z^A \ar[r]_{\underline z^x} &Z^B \ar[r]_{\underline z^y} &Z^C
}
$$
in $\B/\s$, where $\xymatrix@C=0.6cm@R0.6cm{Z^A\ar[r]^{z^x} &Z^B\ar[r]^{z^y} &Z^C\ar@{-->}[r] &}$ is an $\EE$-triangle in $\Z$.
\end{prop}

\begin{proof}
By Lemma \ref{U} and Lemma \ref{V}, we can assume that $A\in \U$ and $C\in \V$. Since $A$ admits an $\EE$-triangle $\xymatrix@C=0.6cm@R0.6cm{A\ar[r] &Z^A \ar[r] &X^A \ar@{-->}[r] &}$ with $Z^A\in \Z$ and $X^A\in \X$, we have the following commutative diagram
$$\xymatrix@C=0.8cm@R0.7cm{
A \ar[r]^x \ar[d] &B \ar[r]^y \ar[d] &C \ar@{=}[d] \ar@{-->}[r]& \\
Z^A \ar[r] \ar[d] &B' \ar[d] \ar[r] &C \ar@{-->}[r] &. \\
X^A  \ar@{=}[r] \ar@{-->}[d] &X^A \ar@{-->}[d]\\
&&
}
$$
Since $C$ admits an $\EE$-triangle $\xymatrix@C=0.6cm@R0.6cm{Y^C\ar[r] &Z^C \ar[r] &C\ar@{-->}[r] &}$ with $Z^C\in \Z$ and $Y^C\in \Y$, we have the following commutative diagram
$$\xymatrix@C=0.8cm@R0.7cm{
&Y^C \ar[d] \ar@{=}[r] &Y^C \ar[d]\\
Z^A \ar[r]^{z^x} \ar@{=}[d] &Z^B \ar[r]^{z^y} \ar[d] &Z^C \ar[d] \ar@{-->}[r] &\\
Z^A \ar[r]  &B' \ar@{-->}[d] \ar[r] &C \ar@{-->}[d] \ar@{-->}[r] &\\
&&
}
$$
with $Z^B\in \Z$. Then we get an isomorphism of triangles
$$\xymatrix@C=0.8cm@R0.7cm{
A\ar[r]^{\underline x} \ar[d]^{\simeq} &B\ar[r]^{\underline y} \ar[d]^{\simeq} &C\ar[d]^{\simeq} \ar[r] &A[1] \ar[d]^{\simeq}\\
Z^A \ar[r]_{\underline z^x} &Z^B \ar[r]_{\underline z^y} &Z^C \ar[r] &Z^A[1].
}
$$
Hence there exists an isomorphism of $\EE$-triangles in $\B/\s$:
$$\xymatrix@C=0.8cm@R0.7cm{
A\ar[r]^{\underline x} \ar[d]^{\simeq} &B\ar[r]^{\underline y} \ar[d]^{\simeq} &C\ar[d]^{\simeq}\\
Z^A \ar[r]_{\underline z^x} &Z^B \ar[r]_{\underline z^y} &Z^C.
}
$$

\end{proof}

\begin{rem}
We can only get that $Z^A[1]$ lies in $\U$ (not necessarily in $\Z$).
\end{rem}


\begin{thm}
Assume that $(\X,\V), (\U,\Y)$ are hereditary cotorsion pairs. Then $F:\Z/[\W]\to \B/\s$ becomes a triangle equivalence.
\end{thm}

\begin{proof}
By Proposition \ref{Fro}, we know that $\Z/[\W]$ is a triangulated category with shift functor $\langle 1\rangle$. Moreover, by Proposition \ref{Here}, we can find that $\s=\s_R=\s_L$ is a thick subcategory, and $((\X,\V),(\U,\Y))$ is a Hovey twin cotorsion pair.
By definition, if $Z\in \Z$, then $Z[1]=Z\langle1\rangle$ in $\B/\s$. For an $\EE$-triangle $A\xrightarrow{f} B\xrightarrow{g} C\dashrightarrow$ in $\Z$,  the image of the induced triangle $A \xrightarrow{\overline f} B\xrightarrow{\overline g} C \xrightarrow{\overline h} A\langle1\rangle$ by equivalent functor $F$ becomes a standard triangle $A \xrightarrow{\underline f} B\xrightarrow{\underline g} C\xrightarrow{\underline h} A[1]$ in $\B/\s$. Moreover, $F\circ \langle1\rangle=[1]\circ F$, this shows that $F\colon\Z/[\W]\to \B/\s$ is a triangle equivalence.
\end{proof}

\begin{rem}
We can find similar results in \cite[Theorem 6.21(2), Proposition 7.14]{S} in the context of Hovey's abelian model structures. Although the
terminology twin cotorsion pair was not used, but hereditary
twin cotorsion pairs were actually considered as well, see also \cite[Definition 6.19]{S}. A connection
to ideal quotients is also considered in \cite[Lemma 6.16, Theorem 6.7(1)]{S}. In fact, Theorem \ref{main} strengthens \v{S}\v{t}ov\'{i}\v{c}ek's result even for abelian cases.
\end{rem}

We give an example of our main results.

\begin{exm}
Let $Q$ be the following infinite quiver:
$$\xymatrix@C=0.5cm@R0.5cm{\cdots \ar[r]^{x_{-5}} &-4 \ar[r]^{x_{-4}} &-3 \ar[r]^{x_{-3}} &-2 \ar[r]^{x_{-2}} &-1 \ar[r]^{x_{-1}} &0 \ar[r]^{x_{0}} &1 \ar[r]^{x_1} &2 \ar[r]^{x_2} &3 \ar[r]^{x_3} &4 \ar[r]^{x_4} &\cdots}$$
Let $\Lambda=kQ/(x_ix_{i+1}x_{i+2},i\neq 4k,x_{4k}x_{4k+1})$. Then the AR-quiver of $\B=\mod \Lambda$ is the following:
{\small 
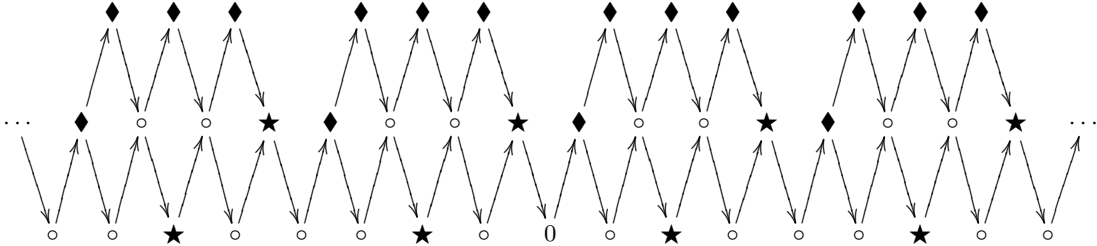
$$\xymatrix@C=0.001cm@R1cm{
 &&&\blacklozenge \ar[dr] &&\blacklozenge \ar[dr] &&\blacklozenge \ar[dr] &&&&\blacklozenge \ar[dr] &&\blacklozenge \ar[dr] &&\blacklozenge \ar[dr] &&&&\blacklozenge \ar[dr] &&\blacklozenge \ar[dr] &&\blacklozenge \ar[dr] &&&&\blacklozenge \ar[dr] && \blacklozenge \ar[dr] && \blacklozenge \ar[dr]\\
\cdots \ar[dr] &&\blacklozenge \ar[dr] \ar[ur] &&\circ \ar[ur] \ar[dr] &&\circ  \ar[ur] \ar[dr] &&\bigstar \ar[dr] &&\blacklozenge \ar[ur] \ar[dr] &&\circ \ar[ur] \ar[dr] &&\circ \ar[ur] \ar[dr] &&\bigstar  \ar[dr] &&\blacklozenge \ar[ur] \ar[dr] &&\circ \ar[ur] \ar[dr] &&\circ \ar[ur] \ar[dr] &&\bigstar \ar[dr] &&\blacklozenge \ar[ur] \ar[dr] &&\circ \ar[dr] \ar[ur] &&\circ \ar[dr] \ar[ur] &&\bigstar \ar[dr] &&\cdots \\
&\circ \ar[ur] &&\circ \ar[ur] &&\bigstar \ar[ur] &&\circ \ar[ur]  &&\circ \ar[ur] &&\circ\ar[ur] &&\bigstar \ar[ur] &&\circ \ar[ur] &&0 \ar[ur] &&\circ \ar[ur] &&\bigstar \ar[ur] &&\circ \ar[ur]  &&\circ \ar[ur] &&\circ \ar[ur] &&\bigstar \ar[ur] &&\circ \ar[ur] &&\circ \ar[ur]
}
$$}
The additive closure of the indecomposable objects denoted by $\bigstar$ and $\blacklozenge$ form a thick subcategory in $\B$,  which is denoted by $\s$. We denote the additive closure of the indecomposable objects in $\blacklozenge$ by $\X$, it is the subcategory of all the projective objects in $\B$. Then $(\X,\s)$ is a cotorsion pair in $\s$,  and we also have two hereditary cotorsion pairs $(\X,\B)$ and $({^{\bot_1}}\s,\s)$. In this case, $\Z={^{\bot_1}}\s$ and $\W=\X$. The indecomposable objects in $({^{\bot_1}}\s)/[\X]$ are denoted by $\spadesuit$ in the diagram:
{\small 
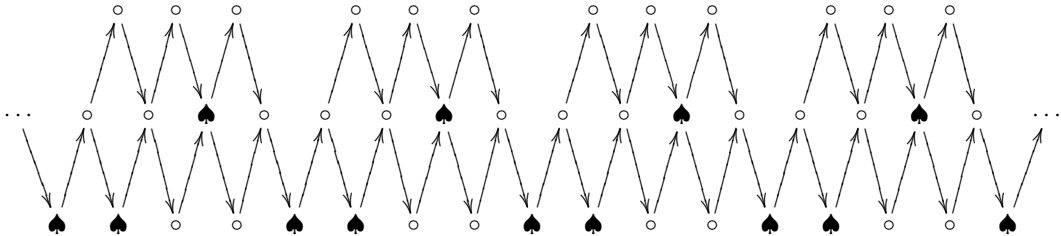
$$\xymatrix@C=0.001cm@R1cm{
 &&&\circ \ar[dr] &&\circ \ar[dr] &&\circ \ar[dr] &&&&\circ \ar[dr] &&\circ \ar[dr] &&\circ \ar[dr] &&&&\circ \ar[dr] &&\circ \ar[dr] &&\circ \ar[dr] &&&&\circ \ar[dr] && \circ \ar[dr] && \circ \ar[dr]\\
\cdots \ar[dr] &&\circ \ar[dr] \ar[ur] &&\circ \ar[ur] \ar[dr] &&\spadesuit \ar[ur] \ar[dr] &&\circ \ar[dr] &&\circ \ar[ur] \ar[dr] &&\circ \ar[ur] \ar[dr] &&\spadesuit \ar[ur] \ar[dr] &&\circ  \ar[dr] &&\circ \ar[ur] \ar[dr] &&\circ \ar[ur] \ar[dr] &&\spadesuit\ar[ur] \ar[dr] &&\circ \ar[dr] &&\circ \ar[ur] \ar[dr] &&\circ \ar[dr] \ar[ur] &&\spadesuit \ar[dr] \ar[ur] &&\circ \ar[dr] &&\cdots \\
&\spadesuit \ar[ur] &&\spadesuit \ar[ur] &&\circ\ar[ur] &&\circ \ar[ur]  &&\spadesuit \ar[ur] &&\spadesuit \ar[ur] &&\circ \ar[ur] &&\circ \ar[ur] &&\spadesuit \ar[ur] &&\spadesuit  \ar[ur] &&\circ\ar[ur] &&\circ \ar[ur]  &&\spadesuit \ar[ur] &&\spadesuit \ar[ur] &&\circ \ar[ur] &&\circ \ar[ur] &&\spadesuit \ar[ur]
}
$$}
The indecomposable objects denoted by $\clubsuit$ are non-zero objects in $\B/\s$ which do not lie in ${^{\bot_1}}\s$, they are isomorphic to the objects in ${^{\bot_1}}\s$:
{\small 
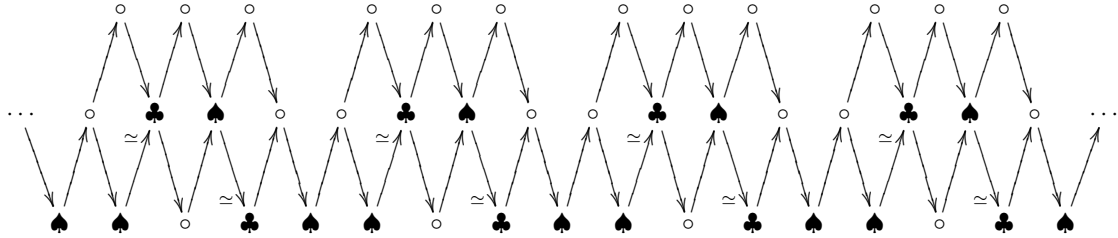
$$\xymatrix@C=0.001cm@R1cm{
 &&&\circ \ar[dr] &&\circ \ar[dr] &&\circ \ar[dr] &&&&\circ \ar[dr] &&\circ \ar[dr] &&\circ \ar[dr] &&&&\circ \ar[dr] &&\circ \ar[dr] &&\circ \ar[dr] &&&&\circ \ar[dr] && \circ \ar[dr] && \circ \ar[dr]\\
\cdots \ar[dr] &&\circ \ar[dr] \ar[ur] &&\clubsuit \ar[ur] \ar[dr] &&\spadesuit \ar[ur] \ar[dr]_(.75){\simeq} &&\circ \ar[dr] &&\circ \ar[ur] \ar[dr] &&\clubsuit \ar[ur] \ar[dr] &&\spadesuit \ar[ur] \ar[dr]_(0.75){\simeq} &&\circ  \ar[dr] &&\circ \ar[ur] \ar[dr] &&\clubsuit \ar[ur] \ar[dr] &&\spadesuit\ar[ur] \ar[dr]_(.75){\simeq} &&\circ \ar[dr] &&\circ \ar[ur] \ar[dr] &&\clubsuit \ar[dr] \ar[ur] &&\spadesuit \ar[dr]_(.75){\simeq} \ar[ur] &&\circ \ar[dr] &&\cdots \\
&\spadesuit \ar[ur] &&\spadesuit \ar[ur]^(.75){\simeq} &&\circ\ar[ur] &&\clubsuit \ar[ur]  &&\spadesuit \ar[ur] &&\spadesuit \ar[ur]^(.75){\simeq} &&\circ \ar[ur] &&\clubsuit \ar[ur] &&\spadesuit \ar[ur] &&\spadesuit  \ar[ur]^(.75){\simeq} &&\circ \ar[ur] &&\clubsuit \ar[ur]  &&\spadesuit \ar[ur] &&\spadesuit \ar[ur]^(0.75){\simeq} &&\circ \ar[ur] &&\clubsuit \ar[ur] &&\spadesuit \ar[ur]
}
$$}
\end{exm}

\vspace{1cm}
\section*{Acknowledgments}

The authors would like to thank the reviewers for their valuable suggestions on mathematical content and English expressions.

\vspace{1cm}

\hspace{-4mm}\textbf{Data Availability}\hspace{2mm} Data sharing not applicable to this article as no datasets were generated or analysed during
the current study.
\vspace{2mm}

\hspace{-4mm}\textbf{Conflict of Interests}\hspace{2mm} The authors declare that they have no conflicts of interest to this work.

\end{document}